\documentclass[11pt,reqno]{amsart}
\usepackage{amsthm}
\usepackage{amssymb}
\usepackage{latexsym}
\usepackage{multicol}
\usepackage{verbatim,enumerate}
\usepackage{hyperref}
\usepackage{amsmath, amscd}
\usepackage{soul}

\advance\textwidth by 1.2in \advance\oddsidemargin by -.6in \advance\evensidemargin by -.6in
\usepackage{tikz}
\usetikzlibrary{matrix,arrows}
\tikzstyle{vertex}=[ circle, fill, draw, inner sep=0pt, minimum size=4pt,]
\tikzstyle{edge}= [thick]

\newtheorem*{cor}{Corollary}%[section]
\newtheorem*{lem}{Lemma}
\newtheorem*{prop}{Proposition}

\newtheorem*{ex}{Example}

\theoremstyle{definition} 
\theoremstyle{definition}
\newtheorem{thm}{Theorem}
\newtheorem*{thm*}{Theorem}

\newtheorem*{rem}{Remark}

\newenvironment{pf}{\proof}{\endproof}
\newcounter{cnt}
\newenvironment{enumerit}{\begin{list}{{\hfill\rm(\roman{cnt})\hfill}}{%
\settowidth{\labelwidth}{{\rm(iv)}}\leftmargin=\labelwidth%
\advance\leftmargin by \labelsep\rightmargin=0pt\usecounter{cnt}}}{\end{list}} \makeatletter
\def\mydggeometry{\makeatletter\dg@YGRID=1\dg@XGRID=20\unitlength=0.003pt\makeatother}
\makeatother \theoremstyle{remark}

\numberwithin{equation}{section}

 \DeclareMathOperator{\Ht}{ht}

\newcommand{\ti}[1]{\widetilde{#1}}

\begin{document}

\newcommand{\thmref}[1]{Theorem~\ref{#1}}
\newcommand{\secref}[1]{Section~\ref{#1}}
\newcommand{\lemref}[1]{Lemma~\ref{#1}}
\newcommand{\propref}[1]{Proposition~\ref{#1}}
\newcommand{\corref}[1]{Corollary~\ref{#1}}
\newcommand{\remref}[1]{Remark~\ref{#1}}
\newcommand{\defref}[1]{Definition~\ref{#1}}
\newcommand{\er}[1]{(\ref{#1})}
\newcommand{\id}{\operatorname{id}}
\newcommand{\ord}{\operatorname{\emph{ord}}}
\newcommand{\sgn}{\operatorname{sgn}}
\newcommand{\wt}{\operatorname{wt}}
\newcommand{\tensor}{\otimes}
\newcommand{\from}{\leftarrow}
\newcommand{\nc}{\newcommand}
\newcommand{\rnc}{\renewcommand}
\newcommand{\dist}{\operatorname{dist}}
\newcommand{\qbinom}[2]{\genfrac[]{0pt}0{#1}{#2}}
\nc{\cal}{\mathcal} \nc{\goth}{\mathfrak} \rnc{\bold}{\mathbf}
\renewcommand{\frak}{\mathfrak}
\newcommand{\supp}{\operatorname{supp}}
\newcommand{\Irr}{\operatorname{Irr}}
\newcommand{\psym}{\mathcal{P}^+_{K,n}}
\newcommand{\psyml}{\mathcal{P}^+_{K,\lambda}}
\newcommand{\psymt}{\mathcal{P}^+_{2,\lambda}}
\renewcommand{\Bbb}{\mathbb}
\nc\bomega{{\mbox{\boldmath $\omega$}}} \nc\bpsi{{\mbox{\boldmath $\Psi$}}}
 \nc\balpha{{\mbox{\boldmath $\alpha$}}}
 \nc\bbeta{{\mbox{\boldmath $\beta$}}}
 \nc\bpi{{\mbox{\boldmath $\pi$}}}
  \nc\bvarpi{{\mbox{\boldmath $\varpi$}}}
\newcommand{\tii}{\ti{\cal I}}
\nc\bepsilon{{\mbox{\boldmath $\epsilon$}}}

  \nc\bxi{{\mbox{\boldmath $\xi$}}}
\nc\bmu{{\mbox{\boldmath $\mu$}}} \nc\bcN{{\mbox{\boldmath $\cal{N}$}}} \nc\bcm{{\mbox{\boldmath $\cal{M}$}}} \nc\blambda{{\mbox{\boldmath
$\lambda$}}}%\nc\mathbb Nu{{\mbox{\boldmath $\nu$}}}

\newcommand{\Tmn}{\bold{T}_{\lambda^1, \lambda^2}^{\nu}}

\newcommand{\lie}[1]{\mathfrak{#1}}
\newcommand{\ol}[1]{\overline{#1}}
\makeatletter
\def\section{\def\@secnumfont{\mdseries}\@startsection{section}{1}%
  \z@{.7\linespacing\@plus\linespacing}{.5\linespacing}%
  {\normalfont\scshape\centering}}
\def\subsection{\def\@secnumfont{\bfseries}\@startsection{subsection}{2}%
  {\parindent}{.5\linespacing\@plus.7\linespacing}{-.5em}%
  {\normalfont\bfseries}}
\makeatother
\def\subl#1{\subsection{}\label{#1}}
 \nc{\Hom}{\operatorname{Hom}}
  \nc{\mode}{\operatorname{mod}}
\nc{\End}{\operatorname{End}} \nc{\wh}[1]{\widehat{#1}} \nc{\Ext}{\operatorname{Ext}}
 \nc{\ch}{\operatorname{ch}} \nc{\ev}{\operatorname{ev}}
\nc{\Ob}{\operatorname{Ob}} \nc{\soc}{\operatorname{soc}} \nc{\rad}{\operatorname{rad}} \nc{\head}{\operatorname{head}}
\def\Im{\operatorname{Im}}
\def\gr{\operatorname{gr}}
\def\mult{\operatorname{mult}}
\def\Max{\operatorname{Max}}
\def\ann{\operatorname{Ann}}
\def\sym{\operatorname{sym}}
\def\loc{\operatorname{loc}}
\def\Res{\operatorname{\br^\lambda_A}}
\def\und{\underline}
\def\Lietg{$A_k(\lie{g})(\bs_\xiigma,r)$}
\def\res{\operatorname{res}}

 \nc{\Cal}{\cal} \nc{\Xp}[1]{X^+(#1)} \nc{\Xm}[1]{X^-(#1)}
\nc{\on}{\operatorname} \nc{\Z}{{\bold Z}} \nc{\J}{{\cal J}} \nc{\C}{{\bold C}} \nc{\Q}{{\bold Q}}
\renewcommand{\P}{{\cal P}}
\nc{\N}{{\Bbb N}} \nc\boa{\bold a}  \nc\bob{\bold b} \nc\boc{\bold c} \nc\bod{\bold d} \nc\boe{\bold e} \nc\bof{\bold f} \nc\bog{\bold g}
\nc\boh{\bold h} \nc\boi{\bold i} \nc\boj{\bold j} \nc\bok{\bold k} \nc\bol{\bold l} \nc\bom{\bold m} \nc\bon{\bold n} \nc\boo{\bold o}
\nc\bop{\bold p} \nc\boq{\bold q} \nc\bor{\bold r} \nc\bos{\bold s} \nc\boT{\bold t} \nc\boF{\bold F} \nc\bou{\bold u} \nc\bov{\bold v}
\nc\bow{\bold w} \nc\boz{\bold z} \nc\boy{\bold y} \nc\ba{\bold A} \nc\bb{\bold B} \nc\bc{\mathbb C} \nc\bd{\bold D} \nc\be{\bold E} \nc\bg{\bold
G} \nc\bh{\bold H} \nc\bi{\bold I} \nc\bj{\bold J} \nc\bk{\bold K} \nc\bl{\bold L} \nc\bm{\bold M}  \nc\bo{\bold O} \nc\bp{\bold
P} \nc\bq{\bold Q} \nc\br{\bold R} \nc\bs{\bold S} \nc\bt{\bold T} \nc\bu{\bold U} \nc\bv{\bold V} \nc\bw{\bold W} \nc\bx{\bold
x} \nc\KR{\bold{KR}} \nc\rk{\bold{rk}} \nc\het{\text{ht }}
\nc\bz{\mathbb Z}
\nc\bn{\mathbb N}

\nc\toa{\tilde a} \nc\tob{\tilde b} \nc\toc{\tilde c} \nc\tod{\tilde d} \nc\toe{\tilde e} \nc\tof{\tilde f} \nc\tog{\tilde g} \nc\toh{\tilde h}
\nc\toi{\tilde i} \nc\toj{\tilde j} \nc\tok{\tilde k} \nc\tol{\tilde l} \nc\tom{\tilde m} \nc\ton{\tilde n} \nc\too{\tilde o} \nc\toq{\tilde q}
\nc\tor{\tilde r} \nc\tos{\tilde s} \nc\toT{\tilde t} \nc\tou{\tilde u} \nc\tov{\tilde v} \nc\tow{\tilde w} \nc\toz{\tilde z} \nc\woi{w_{\omega_i}}
\nc\chara{\operatorname{Char}}
\def\a{\alpha}
\def\l{\lambda}
\def\d{\delta}

\title[Borel--de Siebenthal pairs, Global  Weyl modules and Stanley--Reisner rings]{Borel--de Siebenthal pairs, Global  Weyl modules and Stanley--Reisner rings}
\author{Vyjayanthi Chari}
\address{Department of Mathematics, University of California, Riverside, CA 92521}
\email{chari@math.ucr.edu}
\thanks{V.C. was partially supported by DMS 1303052.}

\author{Deniz Kus}
\address{Mathematisches Institut, Endenicher Allee 60, D-53115 Bonn}
\email{dkus@math.uni-bonn.de}
\thanks{D.K was partially funded under the Institutional Strategy of the University of Cologne within the German Excellence Initiative.}

\author{Matt Odell}
\address{Department of Mathematics, University of California, Riverside, CA 92521}
\email{mattodell10@gmail.com}
\thanks{}

\subjclass[2010]{}
\begin{abstract}
We develop the theory of integrable representations for an arbitrary maximal parabolic subalgebra of an affine Lie algebra. We see that such subalgebras can be thought of as arising in a natural way from a Borel--de Siebenthal pair of semisimple Lie algebras. We see that although there are similarities with the represenation thery of the standard maximal parabolic subalgebra there are also very interesting and non--trivial differences; including the fact that there are examples of non--trivial global Weyl modules which are irreducible and finite--dimensional. We also give a presentation of the endomorphism ring of the global Weyl module; although these are no longer polynomial algebras we see that for certain parabolics these algebras are Stanley--Reisner rings which
are both Koszul and Cohen--Macaualey.
\end{abstract}

\maketitle
%%%%%%%%%%%%%%%%%%%%%%%%%%%%%%%%%%%%%%%%%%%%%%%%%%%%%%%%%%%%%%%%%%%%%%%%%%%%%%%%%%%%%%%%%%%%%%%%%%%%%%%%%%%%%%%%%%%%%%%%%%%%%%%%%%%%%%%%

%%%%%%%%%%%%%%%%%%%%%%%%%%%%%%%%%%%%%%%%%%%%%%%%%%%%%%%%%%%%%%%%%%%%%%%%%%%%%%%%%%%%%%%%%%%%%%%%%%%%%%%%%%%%%%%%%%%%%%%%%%%%%%%%%%%%%%%%

\section{Introduction}

The category of integrable representations of the current algebra $\lie g[t]$ (or equivalently the standard maximal parabolic subalgebra in an untwisted affine Lie algebra) has been intensively studied in recent years. This study has interesting combinatorial consequences and connections with the theory of Macdonald polynomials and their generalizations (see for instance \cite{CI14},  \cite{FeM15},\cite{FM14},). In this paper we develop the corresponding theory for an arbitrary maximal parabolic subalgebra of an untwisted affine Lie algebra. We show that such subalgebras can be realized as the set of fixed points of a finite group action on the current algebra; in other words they are examples of equivariant map algebras as defined in \cite{NSS}.  The representation theory of equivariant map algebras has been developed in \cite{FKKS11,FMS15,NSS}. However much of the theory depends on the group acting freely on $\bc$; in which case it is proved that the representation theory is essentially the same as that of the current algebra.  But this is not true for the non--standard parabolics and there are many interesting and non--trivial differences in the representation theory.

 Recall that two important families of integrable representations of the current algebras are the global and local Weyl modules.  The global Weyl modules are  indexed by dominant integral weights $\lambda \in P^+$ and are universal objects in the category. Moreover the ring of endomorphisms $\ba_\lambda$  in this category is commutative. 
It is known through the  work of \cite{CP01} that $\ba_\lambda$ is a polynomial algebra in a finite number of variables depending on the weight $\lambda$ and that it is  infinite--dimensional if $\lambda\ne 0$.  
The local Weyl modules are indexed by dominant integral weights and maximal ideals in the corresponding algebra $\ba_\lambda $  and are known to be finite--dimensional. The work of \cite{CP01, FoL07, Na11} shows that the dimension of the local Weyl  module depends only on the weight, and not on the choice of maximal ideal in $\ba_\lambda,$ and so the global Weyl module is a free $\ba_\lambda$--module of finite rank.  

In this paper we develop the theory of global and local Weyl modules for an arbitrary maximal parabolic. The modules are indexed by dominant integral weights of a semisimple Lie subalgebra $\lie g_0$ of $\lie g$ which is of maximal rank; a particular example that we use to illustrate all our results is the pair $(B_n, D_n)$ which is also an example of a Borel--de Siebenthal pair. We determine a presentation of $\ba_\lambda$ and show that in general $\ba_{\lambda}$ is not a polynomial algebra and that the corresponding algebraic variety is not irreducible. In fact we give necessary and sufficient conditions on  $\lambda$ for  $\ba_\lambda$ to be  finite--dimensional (we prove that it must be of dimension 1). In  particular the associated global Weyl module is finite--dimensional and under further restrictions on $\lambda$ the global Weyl module is also irreducible. We also show that under suitable conditions on the maximal parabolic the algebra $\ba_\lambda$ is a Stanley--Reisner ring  which is both Koszul and Cohen--Macaualey. 

Finally we study the local Weyl modules associated with a mutiple of a fundamental weight. In this case $\ba_\lambda$ is either one--dimensional or a polynomial algebra.  We determine the dimension of the local Weyl modules and prove that it is independent of the choice of a maximal ideal in $\ba_\lambda$. This proves also that in this case the  global Weyl module is a free $\ba_\lambda$--module of finite rank.
 This fact is false for general $\lambda$ and we give an example of this in Section 7. However, we will show in this example that the global Weyl module is a free module for a suitable quotient algebra of $\ba_{\lambda}$, namely the coordinate ring of one of the irreducible subvarieties of $\ba_{\lambda}$.

 This paper is organized as follows:  In Section 2, we recall a result of Borel and de Siebenthal which realizes all maximal proper semisimple subalgebras, $\lie g_0$,  of maximal rank, of a fixed simple Lie algebra $\lie g$ as the set of fixed points of an automorphism of $\lie g.$  We prove some results on root systems that we will need later in the paper, and discuss the running example of the paper, which is the case where $\lie g$ is of type $B_n$, and $\lie g_0$ is of type $D_n.$

In Section 3 we extend the automorphism of $\lie g$ to an automorphism of $\lie g [t].$ We then study the corresponding equivariant map algebra, which is the set of fixed points of this automorphism.  We discuss ideals of this equivariant map algebra, and show that in this case, the equivariant map algebra is not isomorphic to an equivariant map algebra where the action of the group is free, which makes the representation theory much different from that of the map algebra $\lie g [t]$.  We conclude the section by making the connection between these equivariant map algebras and maximal parabolic subalgebras of the affine Kac-Moody algebra.

In Section 4 we  develop the representation theory of $\lie g[t]^\tau$. Following \cite{CFK10,CIK14}, we define the notion of global Weyl modules, the associated commutative algebra and the local Weyl modules associated to maximal ideals in this algebra. In the case of $\lie g[t]$  it was shown in \cite{CP01}  that  the  commutative algebra associated with a global Weyl module is a  polynomial ring in finitely many variables. This is no longer true for $\lie g[t]^\tau$; however in Section 5 we see that modulo the Jacobson radical, the algebra is a quotient of a finitely generated polynomial ring by a squarefree monomial ideal. By making the connection to Stanley--Reisner theory, we are able to determine the Hilbert series.  In the case when $\boa_j^{\vee}(\alpha_0) = 1$ we also determine the Krull dimension, and we give a sufficient condition for the commutative algebra to be Koszul and Cohen-Macaulay.

In Section 6 we examine an interesting consequence of determining this presentation of the commutative algebra which differs from the case of the current algebra greatly.  More specifically we see that under suitable conditions a global Weyl module can be  finite--dimensional and irreducible, and we give necessary and sufficient conditions for this to be the case.  

We conclude this paper by determining the dimension of the local Weyl module in the case of our running example $(B_n, D_n)$ for multiples of fundamental weights and a few other cases.  We also discuss other features not seen in the case of the current algebra. Namely we give an example of a weight where the dimension of the local Weyl module depends on the choice of maximal ideal in $\ba_{\lambda}$ showing that the global Weyl module is not projective and hence not a free $\ba_{\lambda}$--module.

\textit{Acknowledgements: Part of this work was done when the third author was visiting the University of Cologne. He thanks the University of Cologne for excellent working conditions. He also thanks the Fulbright U.S. Student Program, which made this collaboration possible.}
%%%%%%%%%%%%%%%%%%%%%%%%%%%%%%%%%%%%%%%%%%%%%%%%%%%%%%%%%%%%%%%%%%%%%%%%%%%%%%%%%%%%%%%%%%%%%%%%%%%%%%%%%%%%%%%%%%%%%%%%%%%%%%%%%%%%%%%%

%%%%%%%%%%%%%%%%%%%%%%%%%%%%%%%%%%%%%%%%%%%%%%%%%%%%%%%%%%%%%%%%%%%%%%%%%%%%%%%%%%%%%%%%%%%%%%%%%%%%%%%%%%%%%%%%%%%%%%%%%%%%%%%%%%%%%%%%

\section{The Lie algebras \texorpdfstring{$(\lie g,\lie g_0)$}{(g,g0)}}\label{section1}

\subsection{} We denote the set of complex numbers, the set of integers, non--negative integers, and positive integers  by $\bc$,  $\bz$, $\bz_+$ and $\bn$ respectively. Unless otherwise stated, all the vector spaces considered in this paper are $\bc$-vector spaces and $\otimes$ stands for $\otimes_\bc$. Given any Lie algebra $\lie a$ we let $\bu(\lie a) $ be the universal enveloping algebra of $\lie a$. We also fix an indeterminate $t$ and let $\bc[t]$ and  $\bc[t,t^{-1}]$ be the corresponding  polynomial ring, respectively  Laurent polynomial ring  with complex coefficients.

\subsection{} Let $\lie g$ be a complex simple finite--dimesional Lie algebra of rank $n$  with a fixed Cartan subalgebra $\lie h$. Let $I=\{1,\dots, n\}$ and fix a set $\Delta=\{\alpha_i: i\in I\}$ of simple roots of $\lie g$ with respect to $\lie h$. Let $R$, $R^+$ be the corresponding  set of roots and positive roots respectively. Given $\alpha\in R$ let $\lie g_\alpha$ be the corresponding root space and $a_i$, $i\in I$ be the labels of the Dynkin diagram of $\lie g$; equivalently the highest root of $R^+$ is $\theta=\sum_{i=1}^na_i\alpha_i$.   Fix a Chevalley basis $\{x^\pm _\alpha,  h_i :\alpha\in R^+, i\in I\}$ of $\lie g$, and set $x^\pm_i=x_{\pm\alpha_i}$.  Let $(\ ,\ )$ be the non--degenerate bilinear form on $\lie h^*$ with $(\theta,\theta)=2$ induced by the restriction of the (suitably normalized) Killing form  of $\lie g$ to $\lie h$.  

Let $Q$ be the root lattice with basis $\alpha_i$, $i\in I$.    Define  $\boa_i: Q\to \bz$, $i\in I$  by requiring  $\eta=\sum_{i=1}^n\boa_i(\eta)\alpha_i ,$ and set $\Ht(\eta)=\sum^n_{i=1}\boa_i(\eta)$. For $\alpha\in R$ set $d_\alpha=2/(\alpha,\alpha)$,\  $\boa_i^{\vee}(\alpha)=\boa_i(\alpha)d_\alpha d_{\alpha_i}^{-1}$ and $h_{\alpha}=\sum_{i=1}^n \boa^{\vee}_i(\alpha)h_i$.  
Let $W$ be the Weyl group of $\lie g$ generated by a set of simple reflections $s_i$, $i\in I$ and fix a set of fundamental weights $\{\omega_i: 1\le i\le n\}$ for $\lie g$ with respect to $\Delta$. 

\subsection{} 
\textit{From now on we fix an element $j\in I$ with $a_j\ge 2$ and also fix $\zeta$ to be  a primitive $a_j$--th root of unity. We remark that if $j$ is such that $a_j = 1$ the algebras studied in this paper are known in the literature as current algebras. As we discussed in the introduction, the representation theory of 
current algebras is well-developed and hence we will assume from now on and usually without mention that $a_j\geq 2$.} The following is well--known (see for instance \cite{Hel}). Set $I(j)=I\setminus\{j\}$. 
\begin{prop}\label{defg0}  The assignment $$x_i^\pm \to x_i^\pm,\ \ i\in I(j) ,\ \ x_j^\pm= \zeta_j^{\pm 1}x_j^\pm,$$ defines an automorphism  $\tau:\lie g\to\lie g$ of  order $a_j$.  Moreover, the set of fixed points  $\lie g_0$  is a semismple subalgebra with Cartan subalgebra $\lie h$ and  $$R_0=\{\alpha\in R: \boa_j(\alpha)\in\{0, \pm a_j\}\},$$ is the set of roots of the pair $(\lie g_0,\lie h)$. The set  $\{\alpha_i: i\in I(j)\}\cup\{-\theta\}$ is a simple system for $R_0$.
\hfill\qedsymbol
\end{prop}

 \begin{rem} In the case when $a_j$ is prime, the pair $(\lie g, \lie g_0)$ is an example of a semisimple Borel--de Siebenthal pair.  In other words, $\lie g_0$ is a maximal proper semisimple subalgebra of $\lie g$ of rank $n$. If $a_j$ is not prime we can find a chain of semisimple subalgebras
$$\lie g_0 \subset \boa_1 \subset \cdots\subset \boa_\ell \subset \lie g,$$
such that the successive inclusions are Borel--de Siebenthal pairs.  We shall be interested in infinite--dimensional analogues of these.
\end{rem}

\subsection{} \label{delta0}For our purposes we will need a different simple system for $R_0$ which we choose as follows. The subgroup of $W$ generated by the simple reflections $s_i$, $i\in I(j)$ is the Weyl group of the semisimple Lie algebra  generated by $\{x_i^\pm: i\in I(j)\}$. 
Let $w_\circ$ be the longest element of  this group.
\begin{lem} The set
 $$\Delta_0 =\{\alpha_i: i\in I(j)\}\cup\  \{w_\circ^{-1}\theta\},$$ is a set of simple roots  for $(\lie g_0,\lie h)$ and the corresponding set  $R_0^+$ of positive roots is contained in $R^+$.
\end{lem}
\begin{pf} Since $w_\circ$ is the longest element of the Weyl group generated by $s_i$, $i\in I(j)$, it follows that  for $i\in I(j)$, $$w_\circ\alpha_i\subset\{-\alpha_p: p\in I(j)\}.$$ Hence $$\Delta_0=- w_\circ^{-1}\left(\{\alpha_i: i\in I(j)\}\cup\{-\theta\}\right).$$ Since $w_\circ$ is an element of the Weyl group of $\lie g_0$ it follows from Proposition \ref{defg0} that $\Delta_0$ is a simple system for $R_0$. Moreover $w_\circ^{-1}\theta\in R^+$ since $w_\circ\alpha_j\in R^+$ and $\boa_j(\theta)=a_j$. Hence $\Delta_0\subset R^+$ thus proving the Lemma.
\end{pf} 
 Let $Q_0$ be the weight lattice of $\lie g_0$ determined by $\Delta_0$; clearly $Q_0\subset Q$  and set $Q_0^+=Q_0\cap Q^+$. Then $Q_0^+$ is properly contained in $Q^+$ and we see an example of this at the end of this subsection.

\begin{rem}\label{rem1r} We isolate some immediate consequences of the Lemma which we will  use repeatedly.  From now on we set $\alpha_0=w_\circ^{-1}\theta$,\ $x_{0}^{\pm}=x_{\alpha_0}^{\pm}$ and $h_0=h_{\alpha_0}$.   Then,
\begin{enumerit}
\item[(i)] $\alpha_0$ is a long root.
\item[(ii)] $(\alpha_0,\alpha_i)\le 0$ if $i\in I(j)$ and since  $\alpha_0\in R^+$ it follows that $(\alpha_0,\alpha_j)>0.$
\item[(iii)] $\boa_j(\alpha_0)=a_j$.
\item[(iv)] If $\alpha\in R_0^+$  is such that $\boa_j(\alpha)=a_j$ and $\alpha\ne \alpha_0$, then $\Ht\alpha>\Ht\alpha_0$.
\end{enumerit}
\end{rem}
\begin{ex}
Consider the example of the Borel-de Siebenthal pair $(B_n, D_n),$ so $j=n.$  Recall that the positive roots of $B_n$ are of the form
$$ \alpha_{r,s}:=\alpha_r+\cdots+\alpha_s,\quad \alpha_{r,\overline{s}}:=\alpha_r+\cdots+\alpha_{s-1}+2\alpha_s+\cdots+2\alpha_n.$$  Moreover, $\theta = \alpha_{1,\overline{2}}$ and so $a_n = 2.$ In this case, $\lie g_0$ is of type $D_n$ and $\alpha_0=\alpha_{n-1}+2\alpha_n$. The simple system for $D_n$ is $\Delta_0 = \{ \alpha_1,\dots,\alpha_{n-2},\alpha_{n-1},\alpha_0\}$ ($\alpha_0$ and $\alpha_{n-1}$ correspond to the spin nodes) and the root system for $D_n$ is the set of all long roots of $B_n.$  We note  that $\alpha_n \in Q^+ \setminus Q^+_0$ as mentioned earlier in this section.
\end{ex}
\subsection{} For $1\le k<a_j$ set $$R_k=\{\alpha\in R: \boa_j(\alpha)\in\{k, -a_j+k\}\},\ \ \lie g_k=\bigoplus_{\alpha\in R_k}\lie g_\alpha.$$ Equivalently $$\lie g_k=\{x\in\lie g: \tau(x)=\zeta^k x\}.$$ Setting $R_k^+=R_k\cap R^+$, we observe that \begin{equation}\label{relok} [x_{0}^+, R_k^+]=0,\ \ 1\le k<a_j.\end{equation}
\begin{prop}\label{facts} We have,

\begin{enumerit}
\item [(i)] $\lie g_0=[\lie g_{1}, \lie g_{a_j-1}]$.
\item[(ii)] For all $1\leq k<a_j$ the subspace $\lie g_k$ is an irreducible $\lie g_0$--module.
\item [(iii)]For all $0\le m<k< a_j$, we have $\lie g_k=[\lie g_{k-m}, \lie g_m]$.
\end{enumerit}
\end{prop}

\begin{proof} Each connected component of the Dynkin diagram of the semisimple algebra $\lie g_0$ contains some  simple root $\alpha_i$ with  $\alpha_i(h_j)<0$. Since $0\neq h_j=[x^+_{j}, x^{-}_{_j}]\in [\lie g_1, \lie g_{a_j-1}]$ it follows that  $[\lie g_{1}, \lie g_{a_j-1}]$ intersects each simple ideal of $\lie g_0$ non--trivially, which proves (i). 

If $a_j=2$, the proof of the irreducibility in part (ii) of the proposition can be found in \cite[Proposition 8.6]{K90}. 
If  $a_j\geq 3$ then $\lie g$ is of exceptional type and the proof is done in a case by case fashion. One  inspects the set of roots to notice that for $1\le k<a_j$ there exists a unique root $\theta_k\in R^+_k$  such that $\Ht\theta_k$ is maximal. This means that $x_{\theta_k}^+$ generates an irreducible $\lie g_0$--module and a calculation proves that the dimension of this module is precisely $\dim \lie g_k$ 
and establishes part (ii).
Part (iii) is now immediate if we prove that the $\lie g_0$--module $[\lie g_{k-m},\lie g_{m}]$ is non--zero and this is again proved by inspection. We omit the details.
\end{proof}

\label{proptheta}Part (ii) of the proposition implies that  $R_k^+$ has a unique element $\theta_k$ such that  the following holds: \begin{equation}\label{22}(\theta_k,\alpha_i)\ge 0\ \  {\rm{and}} \ [x_i^+, \lie g_{\theta_k}]=0, \ \ \ i\in I(j)\cup\{0\}. \end{equation}  Since $\theta_k\ne\theta$ it is immediate that $$[x_j^+, \lie g_{\theta_k}]\ne 0,\ \ i.e.,\ \  \theta_k+\alpha_j\in R^+.$$ Notice that $x_{\theta_k}^-\in\lie g_{a_j-k}$ and $[x_i^-, x_{\theta_k}^-]=0$ for all $i\in I(j)\cup\{0\}$.
 Moreover  \begin{equation}\label{alli} \boa_i(\theta_k)>0,\ \ i\in I,\ \ 1\le k<a_j.\end{equation}  To see this  note that  the set $\{i:\boa_i(\theta_k)=0\}$ is contained in $I(j)$. Since $R$ is irreducible there must exist $i, p\in I $ with $\boa_i(\theta_k)=0$ and $\boa_p(\theta_k)>0$ and $(\alpha_i,\alpha_p)<0$. It follows that $(\theta_k, \alpha_i)<0$ which contradicts \eqref{22}.  As a consequence of \eqref{alli} we get, 
\begin{equation}\label{diff} (\theta,\theta_k)>0,\ \ 1\le k<a_j,\ \ {\rm{and\  hence}}\ \ \theta-\theta_k\in R^+_{a_j-k}.\end{equation}
 Finally, we note that  since $(\theta_k+\alpha_j,\alpha_0)=(\theta_k,\alpha_0)+(\alpha_j,\alpha_0)>0$ (see the Remark in Section \ref{delta0})  we now have $$\theta_k+\alpha_j-\alpha_0\in R,\ \ k\neq a_j-1,\ \ \theta_{a_j-1}+\alpha_j-\alpha_0\in R_0^+\cup\{0\}.$$

\begin{ex}
In the  case of $(B_n, D_n)$,  we recall $a_n=2$. In this case, $R_1$ is the set of all short roots of $B_n,$ and  $\theta_1=\alpha_{1}+\cdots+\alpha_n$. When $n \geq 4$, $\lie g_1$ is the natural representation of $D_n$.  When $n=3$, $\lie g_1$ is the second fundamental representation of $A_3$.
\end{ex}

%%%%%%%%%%%%%%%%%%%%%%%%%%%%%%%%%%%%%%%%%%%%%%%%%%%%%%
%%%%%%%%%%%%%%%%%%%%%%%%%%%%%%%%%%%%%%%%%%%%%%%%%%%%%

\section{ The algebras \texorpdfstring{$(\lie g[t], \lie g[t]^\tau)$}{(g[t],g[t]{tau})}}\label{section2} In this section we define the current algebra version of the pair $(\lie g,\lie g_0)$; namely we extend the automorphism $\tau$ to the current algebra and study its fixed points. The fixed point algebra is an example of an equivariant map algebra studied in \cite{NSS}. We show that our examples are particularly interesting since they can also be realized as maximal parabolic subalgebras of affine Lie algebras. We also show that our examples never arise from a free action of a finite abelian  group on the current algebra. This fact makes the study of its representation theory quite different from that of the usual current algebra.

\subsection{}\label{grad} Let $\lie g[t]=\lie g\otimes\bc[t]$ be the Lie algebra with the Lie bracket given by extending scalars. Recall the automorphism $\tau: \lie g\to \lie g$ defined in Section \ref{section1}. It  extends to an automorphism of $\lie g[t]$ (also denoted as $\tau$) by $$\tau(x\otimes t^r)=\tau(x)\otimes \zeta^{-r}t^r,\ \ x\in\lie g,\ \ r\in\bz_+.$$ Let $\lie g[t]^\tau$ be the subalgebra of fixed points of $\tau$; clearly $$\lie g[t]^\tau=\bigoplus_{k=0}^{a_j-1}\lie g_k\otimes t^{k}\bc[t^{a_j}].$$ Further, if we regard $\lie g[t]$ as a $\bz_+$--graded Lie algebra by requiring the grade of $x\otimes t^r$ to be $r$ then $\lie g[t]^\tau$ is also a $\bz_+$--graded Lie algebra, i.e.,
$$\lie g[t]^\tau=\bigoplus_{s\in \bz_+} \lie g[t]^\tau[s].$$
A graded representation of $\lie g[t]^\tau$ is a $\bz_+$--graded vector space $V$ which admits a compatible
Lie algebra action of $\lie g[t]^\tau$, i.e.,
$$V=\bigoplus_{s\in \bz_+}V[s],\ \ \lie g[t]^\tau[s] V[r]\subset V[r+s],\ r,s\in \bz_+.$$

\subsection{} Given $z\in\bc$, let $\ev_z: \lie g[t]\to \lie g$ be defined by $\ev_z(x\otimes t^r)=z^rx$, $x\in\lie g$, $r\in\bz_+$. It is easy to see that  \begin{equation}\label{diffquot}\ev_0(\lie g[t]^\tau)=\lie g_0,\ \ \ev_z(\lie g[t]^\tau)=\lie g,\ \ z\ne 0.\end{equation} More generally, one can construct ideals of finite codimension in $\lie g[t]^\tau$ as follows. Let $f\in\bc[t^{a_j}]$ and $0\le k<a_j$. The ideal $\lie g\otimes t^kf\bc[t]$ of $\lie g[t]$ is of finite codimension and  preserved by $\tau$.  Hence, $\lie i_{k,f}=(\lie g\otimes t^kf\bc[t^{a_j}])^{\tau}$ is an ideal of finite codimension in $\lie g[t]^\tau$. Notice that $$\ker \ev_0 \cap \lie g[t]^\tau=\lie i_{1,1},\ \ \ker\ev_z \cap \lie g[t]^\tau=\lie i_{0, (t^{a_j}-z^{a_j})}.$$
 We now prove,
\begin{prop} \label{fincodim} Let $\lie i$ be a non--zero ideal  in $\lie g[t]^\tau$. Then there exists $0\le k<a_j$ and $f\in\bc[t^{a_j}]$ such that $\lie i_{k,f}\subset\lie i$. In particular, any non--zero ideal in $\lie g[t]^\tau$ is of finite codimension.
\end{prop}
\begin{pf} 
For  $0\le k<a_j$,  set $$ S_k= \big\{ g \in \mathbb{C}[t^{a_j}]:  x \otimes t^kg \in \mathfrak{i} \text{ for all }  x \in \lie g_k\big\}. $$   We claim that $S_k$ is an ideal in $\bc[t^{a_j}]$ for all $0\le k<a_j$ and 
\begin{equation}\label{props}t^{a_j} S_{a_j-1} \subset S_0 \subset S_1\subset \cdots \subset S_{a_j-1}. \end{equation}
Let $0\le k< a_j$, $g\in S_k$ and $f\in\bc[t^{a_j}]$. By Proposition \ref{facts} we have $[\lie g_0,\lie g_k]=\lie g_k$ and hence any $x\in\lie g_k$ can be written as  $x=\sum_{s=1}^r[z_s, y_s]$ with $z_s\in\lie g_0$ and $y_s\in\lie g_k$. Therefore $$x\otimes t^kfg=\sum_{s=1}^r[z_s\otimes f, y_s\otimes t^kg]\in\lie i,$$ which proves that $fg\in S_k$ and hence $S_k$ is an ideal in $\mathbb{C}[t^{a_j}]$. A similar argument using $[\lie g_m,\lie g_{k-m}]=\lie g_k$ proves the desired inclusions \eqref{props}.

We now prove that $S_k\ne 0$ for some $0\le k< a_j$. If $\lie i\subset \lie g_0\otimes\bc[t^{a_j}]$ then $[x_j^\pm, \lie i]=0$. This would imply that $(x^+_i\otimes g)\notin\lie i$ for any $i\in I(j)\cup\{0\}$ and $g\in\bc[t^{a_j}]$. This is because if $(x^+_i\otimes g)\in \lie i$  then $(\lie a\otimes g)\in\lie i$ for the  simple ideal $\lie a$ of $\lie g_0$ containing $x^+_i$. But $\lie a$ contains a simple root  vector $x_k^+$ with  $[x_k^+,x_j^+]\ne 0$ and hence we have a contradiction.  In other words we have proved that $\lie i$ must contain an element of the form 
 $(x\otimes t^k g)$ for some root vector $x\in \lie g_k$, $k>0$  and $0\ne g\in \bc[t^{a_j}]$. Since $\lie g_k$ is an irreducible $\lie g_0$--module we have $(\lie g_k\otimes t^k g)\in \lie i$, i.e. $S_k\ne 0$ and we are done. 

Using \eqref{props} we also see that $S_r\neq 0$ for all $0\le r< a_j$; let $f_r\in\bc[t^{a_j}]$  be a non--zero generator for the ideal $S_r$.  By \eqref{props} there exist $g_0,\dots,g_{a_j-1}\in \mathbb{C}[t^{a_j}]$ such that 
$$f_r=g_rf_{r+1},\  0\leq r \leq a_j-2,\ \ t^{a_j}f_{a_j-1}=g_{a_j-1}f_0.$$ This implies
$$g_{a_j-1}f_0=g_0\cdots g_{a_j-1}f_{a_j-1}=t^{a_j}f_{a_j-1}.$$ Hence there exists a unique $m\in\{0,\dots, a_j-1\}$ such that $g_m=t^{a_j}$  and $g_p=1$ if $p\ne m$. Taking $f=f_{m+1}$, where we understand $f_{a_j}=f_{0}$, we see that $$\lie i_{k,f}\subset\lie i,\ \ k=m+1-a_j\delta_{m,a_j-1}.$$
\end{pf}

\subsection{}  We now show that $\lie g[t]^\tau$ is never a current algebra or more generally an equivariant map algebra with free action.  For this, we  recall from \cite{NSS} the definition of an equivariant map algebra.  Thus, let  $\lie a$ be any complex Lie algebra and $A$ a finitely generated commutative associative algebra.  Assume also that  $\Gamma$ is  a finite abelian group acting on $\mathfrak{a}$ by Lie algebra automorphisms and on $A$ by algebra automorphisms. Then we have an induced action on the Lie algebra  $\big(\mathfrak{a}\otimes A\big)$ (the commutator is given in the obvious way)  such that $\gamma(x\otimes f)=\gamma x\otimes \gamma f$.   An  equivariant map algebra  is defined to be the fixed point subalgebra:
$$\big(\mathfrak{a}\otimes A\big)^{\Gamma}:=\big\{z\in \big(\mathfrak{a}\otimes A\big) \mid \gamma(z)=z \ \ \forall  \ \gamma\in \Gamma\big\}.$$ The finite--dimensional irreducible representations of such algebras (and hence for $\lie g[t]^\tau$) were given in \cite{NSS} and generalized  earlier work on affine Lie algebras.  
%We shall  give an independent proof of this result for $\lie g[t]^\tau$  later in this paper. 
In the case when $\Gamma$ acts freely on $A$, many aspects of  the representation theory of the equivariant map algebra are the same as the representation theory of $\lie a\otimes A$ (see for instance \cite{FKKS11}). The importance of the following proposition is now clear.
\begin{prop} \label{eqmap}  The Lie algebra  $\lie g[t]^\tau $  is  not  isomorphic to an  equivariant map algebra $(\mathfrak{a} \otimes A )^{\Gamma}$ with $\lie a$ semisimple  and $\Gamma$ acting freely on $A$.\end{prop}

\begin{pf}
Recall our assumption that $a_j>1$ and assume for a contradiction that $$\lie g[t]^\tau\cong (\lie a\otimes A)^\Gamma$$ where $\lie a$ is semi--simple.
Write $\lie a=\lie a_1\oplus \cdots \oplus \lie a_k$ where each $\lie a_s$ is a  direct sum of copies of a  simple Lie  algebra $\lie g_s$ and $\lie g_s\ncong \lie g_m$ if $m\ne s$. Clearly $\Gamma$ preserves $\lie a_s$ for all $1\le s\le k$ and hence $$\lie g[t]^\tau\cong (\lie a\otimes A)^\Gamma\cong \oplus_{s=1}^k(\lie a_s\otimes A)^\Gamma.$$  Since $\lie g[t]^\tau$ is infinite--dimensional  at least one of the summands $(\lie a_s\otimes A)^\Gamma$ is infinite--dimensional, say $s=1$ without loss of generality.  But this means that $\oplus_{s=2}^k(\lie a_s\oplus A)^\Gamma$ is an ideal  which is not of  finite codimension which contradicts Proposition \ref{fincodim}. Hence we must have $k=1$, i.e. $\lie a=\lie a_1$.
 It was proven in \cite[Proposition 5.2]{NSS} that if $\Gamma$ acts freely on $A$ then any finite--dimensional simple quotient of $(\lie a\otimes A)^\Gamma$  is  a quotient of $\lie a$; in particular in our situation  it follows that all the finite--dimensional simple quotients  of $(\lie a\otimes A)^\Gamma$ are isomorphic. On the other hand, 
\eqref{diffquot}  shows that $\lie g[t]^\tau$ has  both $\lie g_0$ and $\lie g$ as quotients. Since  $\lie g_0$ is not isomorphic to $\lie g$  we have the desired contradiction.
\end{pf}

\subsection{} We now make the connection of $\lie g[t]^\tau$ with a maximal parabolic subalgebra of the untwisted affine Lie algebra $ \widehat{\lie g}$ associated to $\lie g$. 

 Fix a Cartan subalgebra $\wh{\lie h}$ of $\wh{\lie g}$ containing $\lie h$ and recall that $$\wh{\lie h}=\lie h\oplus\bc c\oplus\bc d,$$ where $c$ spans the one--dimensional center of $\wh{\lie g}$ and $d$ is the scaling element.  Let $\delta\in \wh{\lie h}^*$ be the unique non--divisible positive imaginary root, i.e.,  $\delta(d)=1$ and $\delta(\lie h\oplus\bc c)=0$.  Extend $\alpha\in\lie h^*$ to an element of $\wh{\lie h}^*$ by $\alpha(c)=\alpha(d)=0$. The elements $\{\alpha_i: 1\le i\le n\}\cup\{-\theta+\delta\}$ is a set of simple roots for $\wh{\lie g}$.  We define a grading on $\wh{\lie g}$ as follows: for $ r \in \mathbb{Z} $ and  for each $x_{\alpha} \in \wh{\lie g}_{\alpha}$, $x_{\alpha} \in \wh{\lie g}[r]$ iff  $\alpha = \sum_{i=0}^n r_i \alpha_i $ and  $ r = r_j. $ The following is not hard to prove.
\begin{prop}\label{heliso} Let $\wh{\lie p}$ be the maximal parabolic subalgebra generated by the elements $x_i^\pm$, $i\in I(j)$, $x_{\pm(\delta-\theta)}$ and $x_j^+$. Then there exists an isomorphism of graded Lie algebras $$\widehat{\lie p}\cong\lie g[t]^\tau.$$
\hfill\qed
\end{prop}

\begin{ex}
In the case of $(B_n, D_n)$, we have $a_n=2$. Recall that the derived subalgebra of an untwisted affine Lie algebra can be realized as a universal central extension of the loop algebra $\lie g \otimes \mathbb{C}[t, t^{-1}]$.    For $r \in \mathbb{Z}_+$, the elements $x_{\alpha_{i,j}}^{\pm} \otimes t^r,  x_{\alpha_{i,\overline{j+1}}}^{\pm} \otimes t^{(r \mp 1)}$ and $ h_i \otimes t^r$ for $1 \leq i \leq j \leq n$ form a graded basis of $\wh{\lie p}[2r],$ and the elements $x_{\alpha_{i,n}}^{-} \otimes t^{r+1} $ and $x_{\alpha_{i,n}}^{+} \otimes t^{r} $ for $ 1 \leq i \leq n$  form a basis of $\wh{\lie p}[2r+1].$

The map
$$x \otimes t^k \mapsto
x\otimes t^{2r+ s},\ \text{ if } x \otimes t^{k} \in \widehat{\mathfrak{p}}[2r+s]
$$
gives the isomorphism in Proposition \ref{heliso}.
\end{ex}

%%%%%%%%%%%%%%%%%%%%%%%%%%%%%%%%%%%%%%%%%%%%%%%%%%%%%%%%%%%%%%%%%%%%%%%%%%%%%%%%%%%%%%%%%%%%%%%%%%%
%%%%%%%%%%%%%%%%%%%%%%%%%%%%%%%%%%%%%%%%%%%%%%%%%%%%%%%%%%%%%%%%%%%%%%%%%%%%%%%%%%%%%%%%%%%%%%%%%%%
%%%%%%%%%%%%%%%%%%%%%%%%%%%%%%%%%%%%%%%%%%%%%%%%%%%%%%%%%%%%%%%%%%%%%%%%%%%%%%%%%%%%%%%%%%%%%%%%%%%

\section{The category \texorpdfstring{$\ti{\cal I}_{}$}{I} }\label{section3}
In this section we  develop the representation theory of $\lie g[t]^\tau$. Following \cite{CFK10}, \cite{CIK14}, we define the notion of global Weyl modules, the associated commutative algebra and the local Weyl modules associated to maximal ideals in this algebra. In the case of $\lie g[t]$  it was shown in \cite{CP01}  that  the  commutative algebra associated with a global Weyl module  is a  polynomial ring in finitely many  variables. This is no longer true for $\lie g[t]^\tau$; however we shall see that modulo the Jacobson radical, the algebra is a quotient of a finitely generated  polynomial ring by a squarefree monomial ideal. As a consequence we see that under suitable conditions a global Weyl module  can be  finite--dimensional and irreducible. More precise statements can be found in Section~\ref{section5}.

\subsection{}
 Fix a set of fundamental weights $\{\lambda_i: i\in I(j)\cup\{0\}\}$ for $\lie g_0$ with respect to $\Delta_0$ and let $P_0, P_0^+$ be their $\bz$ and $\bz_+$--span respectively.  Note that the subset $$P^+= \{\lambda\in P_0^+:\lambda(h_j)\in\bz_+\}$$ is precisely the set of dominant integral weights for $\lie g$ with respect to $\Delta$. Also note that $P^+ $ is properly contained in $P_0^+$.  For example, in the $B_n$ case, $\lambda_{n-1} \in P_0^+$, and $\lambda_{n-1}(h_n) = -1$.  It is the existence of these types of weights that cause the representation theory of $\lie g [t]^{\tau}$ to be different from that of $\lie g[t]$.  

For $\lambda\in P_0^+$ let $V_{\lie g_0}(\lambda)$ be the irreducible finite--dimensional $\lie g_0$--module with highest weight $\lambda$ and highest weight vector $v_\lambda$;  if $\lambda\in P^+$ the module $V_{\lie g}(\lambda)$ and the vector $v_\lambda$  are defined in the same way.

\subsection{}\label{charh}  Let $\ti{\cal I}$ be the category whose objects are $\lie g[t]^\tau$--modules with the property  that  they are $\lie g_0$ integrable and where the morphisms are $\lie g[t]^\tau$--module maps. In other words an object $V$ of $\ti{\cal I}$ is a $\lie g[t]^\tau$--module which   is isomorphic to  a direct sum of  finite--dimensional $\lie g_0$--modules. It follows that $V$ admits a weight space decomposition $$V=\bigoplus_{\mu\in P_0} V_\mu,\ \ V_\mu=\{v\in V: hv=\mu(h)v,\ h\in\lie h\},$$ and we set $\wt V=\{\mu\in P_0: V_\mu\ne 0\}.$   Note that $$w\wt V\subset\wt V,\ \  w\in W_0,$$ where $W_0$ is the Weyl group of $\lie g_0$. 
%If $\wt V$ is a finite set, we define the $\lie h$--character of $V$ by

%$$\text{ch}  V=\sum_{\mu\in P_0} \dim V_\mu \ e(\mu)\in \bz[P_0].$$

For  $\lambda\in P_0^+$ we let $\ti{\cal I}^\lambda$ be the full subcategory of $\tii$ whose objects $V$ satisfy the condition that $\wt V\subset \lambda-Q^+$; note that this is a weaker condition than requiring the set of weights be contained in $\lambda-Q_0^+$ (see   \secref{delta0}).  

\begin{lem}\label{wfinite}  Suppose that $V$ is an object of $\ti{\cal I}^\lambda$ and let $\mu\in\wt V$ and $\alpha\in R^+$. Then $\mu-s\alpha\in\wt V$ for only finitely many $s$.

\end{lem} 
\begin{pf} If $\alpha\in R_0^+$ the result is immediate since $V$ is a sum of finite--dimensional $\lie g_0$--modules. 
Since $\alpha\in P_0$, it follows that there exists $w\in W_0$ such that $w\alpha$ is in the anti--dominant chamber for the action of $W_0$ on $\lie h$.  This implies that $w\alpha=-r_0\alpha_0-\sum_{i\in  I(j)} r_i\alpha_i$ where the $r_i$ are non--negative {\em rational } numbers.  Since $W_0$ is a subgroup of $W$ it  follows that $-w\alpha\in R^+$. This shows that  if  $\mu\in\wt V$ is such that $\mu-s\alpha\in\wt V$, then $$w\mu-sw\alpha\in\wt V\subset\lambda-Q^+.$$ This is possible only for finitely many $s$ and hence the Lemma is established.
\end{pf}

\subsection{} Let $$\lie g=\lie n^-\oplus\lie h\oplus\lie n ^+,\ \ \  \lie n^\pm=\bigoplus_{\alpha\in  R^+}\lie g_{\pm \alpha},$$ be the triangular decomposition of $\lie g$.   Since $\tau$ preserves the subalgebras $\lie n^\pm$ and $\lie h$ we have $$\lie g[t]^\tau=\lie n^-[t]^\tau\oplus\lie h[t]^\tau\oplus\lie n^+[t]^\tau.$$  Further $\lie h[t]^\tau\cong\lie h\otimes \bc[t^{a_j}]$ is a commutative subalgebra of $\lie g[t]^\tau$.

For $\lambda\in P_0^+$ the global Weyl module $W(\lambda)$ is the cyclic $\lie g[t]^\tau$--module generated by an element $w_\lambda$ with defining relations: for $h\in\lie h$ and $i\in I(j)\cup\{0\}$, \begin{equation}\label{glweyl}hw_\lambda =\lambda(h)w_\lambda,\   \  \lie n^+[t]^\tau w_\lambda=0,\ \    (x^-_{i}\otimes 1)^{\lambda(h_{i})+1}w_\lambda=0.\end{equation}
It is elementary to check  that $W(\lambda)$ is an object of $\tii_j^\lambda$, one just needs to observe that the elements $x^{\pm}_{i}$, $i\in I(j)\cup\{0\}$ act locally nilpotently on $W(\lambda)$.   Moreover, if we declare the grade of $w_\lambda$ to be zero then $W(\lambda)$ acquires the structure of a $\bz_+$ graded $\lie g[t]^\tau$--module.

\subsection{}   As in \cite{CFK10} one checks easily that the following formula  defines a right action of $\lie h[t]^\tau$ on $W(\lambda)$: $$(uw_\lambda)a=uaw_\lambda,\ \  u\in \bu(\lie g[t]^\tau),\ \ a\in\lie h[t]^\tau.$$ Moreover this action commutes with the left action of $\lie g[t]^\tau$. In particular, if we set $$\ann_{ \lie h[t]^\tau}(w_\lambda)=\{a\in \bu(\lie h[t]^\tau): aw_\lambda=0\},\qquad \ba_\lambda=\bu(\lie h[t]^\tau)/\ann_{ \lie h[t]^\tau}(w_\lambda), $$ we get that $\ann_{ \lie h[t]^\tau}(w_\lambda)$ is an ideal in $\bu(\lie h[t]^\tau)$ and that $W(\lambda)$ is a bi--module for $(\lie g[t]^\tau,\ba_\lambda)$.  
It is clear that $\ann_{\lie h[t]^\tau}(w_\lambda)$ is a graded ideal   of $\bu(\lie h[t]^\tau)$ and hence the algebra $\ba_\lambda$ is a $\bz_+$--graded algebra with a unique graded maximal ideal $\bi_0$.

 It is obvious from the definition that  we have an isomorphism of right $\ba_\lambda$--modules $W(\lambda)_\lambda\cong \ba_\lambda$.  We now prove,
\begin{prop} \label{fingen} For all $\lambda\in P_0^+$ the algebra $\ba_\lambda$ is finitely generated and $W(\lambda)$ is a finitely generated $\ba_\lambda$--module.
\end{prop} 
The proof of the proposition is very similar to the one given in \cite{CFK10} but we sketch the proof below for the reader's convenience and also to set up some further necessary notation. Unlike in the case of $\lie g[t]$ we will later see that the global Weyl module is not a free $\ba_\lambda$ module in general (see \secref{notfreealambda})

\subsection{} We need an additional result to prove Proposition \ref{fingen}. For $\alpha
\in R^+$ and $r\in\bz_+$, define elements $P_{\alpha,r}\in\bu(\lie h[t]^\tau)$ recursively  by
$$ P_{\alpha,0}=1,\ \ P_{\alpha
,r}=-\frac{1}{r}\sum_{p=1}^{r}(h_{\alpha}\otimes t^{a_jp})P_{\alpha, r-p},\ \ r\geq 1.  $$  Equivalently $P_{\alpha, r}$ is the coefficient of $u^r$ in the formal power series $$P_{\alpha}(u)=\exp\left(-\sum_{r\ge 1} \frac{h_{\alpha}\otimes t^{a_jr}}{r}u^r\right).$$  Writing $h_\alpha=\sum_{i=1}^n\boa_i^\vee(\alpha) h_i$, we  see that
$$P_{\alpha}(u)=\prod_{i=1}^n P_{\alpha_i}(u)^{\boa_i^\vee(\alpha)},\ \ \alpha\in R^+.$$
Set $P_{\alpha_i, r}= P_{i,r}$, $i\in I\cup\{0\}$. The following is now trivial from the Poincar\'{e}--Birkhoff--Witt theorem.
\begin{lem}\label{poly}
 The algebra  $\bu(\lie h[t]^\tau)$ is the polynomial algebra in the variables $$\{P_{i, r}: i\in I(j)\cup\{0\},\ r\in\bn\},$$  and also in the variables $$\{P_{i, r}: i\in I,\ r\in\bn\}.$$ \hfill\qedsymbol \end{lem}
 The comultiplication $\tilde{\Delta}: \bu(\lie g[t]^\tau)\to\bu(\lie g[t]^\tau)\otimes \bu(\lie g[t]^\tau)$ satisfies 
\begin{equation}\label{grouplike}\tilde{\Delta}(P_\alpha(u))=P_\alpha(u)\otimes P_\alpha(u),\ \ \alpha\in R^+.\end{equation}
 For $x\in \bu(\lie g[t]^\tau)$, $r\in\bz_+$, set 
$$x^{(r)}=\frac{1}{r!}x^r.$$

\subsection{} The following can be found in \cite[Lemma 1.3]{CP01} and is a reformulation of a result of Garland, \cite{G78}. 
\begin{lem} Let $x^\pm, h$ be the standard basis of $\lie{sl}_2$ and let $V$ be a representation of the subLie algebra  of $\lie{sl}_2[t]$ generated by $(x^+\otimes 1)$ and $(x^-\otimes t)$. Assume that $0\ne v\in V$ is such that $(x_\alpha^+\otimes t^{r})v=0$ for all $r\in\bz_+$. 
 For all $r\in\bz_+$ we have
   \begin{equation}\label{gar1}(x^+\otimes 1)^{(r)}(x^-\otimes t)^{(r)}v=(x^+\otimes t)^{(r)}(x^-\otimes 1)^{(r)}v= (-1)^r P_{r} v,\end{equation} where $$ \sum_{r\ge 0}{P}_ru^{r}=\exp\left(-\sum_{r\ge 1}\frac{h\otimes t^{r}}{r}u^{r}\right).$$
Further, \begin{equation}\label{gar2}(x^+\otimes 1)^{{(r)}}(x^-\otimes t)^{{(r+1)}}v =(-1)^r\sum_{s=0}^{r}(x^-\otimes t^{s+1}){P}_{ r-s}v.\end{equation}\hfill\qedsymbol\end{lem}

\subsection{Proof of Proposition \ref{fingen} }\label{prooffingen} Given $\alpha\in R^+$, it is easily seen that the elements $(x_\alpha^+\otimes t^{\boa_j(\alpha)})$ and $(x_\alpha^-\otimes t^{a_j-\boa_j(\alpha)})$ generate a subalgebra of $\lie g[t]^\tau$ which is  isomorphic 
 to the subalgebra   of $\lie{sl}_2[t]$ generated by $(x^+\otimes 1)$ and $(x^-\otimes t)$. Using the defining relations of $W(\lambda)$ and equation  \eqref{gar1} we  get that \begin{equation}\label{pgamma} P_{\alpha, r}w_\lambda=0,\ \ r\ge \lambda(h_\alpha)+1,\ \ \alpha\in R^+_0. \end{equation}
It also follows from Lemma \ref{wfinite} that $P_{j, r}w_\lambda=0$ for all $r>>0$. Using Lemma \ref{poly} we see that $\ba_\lambda$ is finitely generated by the images of the elements $$\{P_{i,r}: i\in I(j)\cup\{0\}, \  \ r\leq \lambda(h_i)\}.$$

We now prove that $W(\lambda)$ is a finitely generated $\ba_\lambda$--module. Fix an enumeration $\beta_1,\dots ,\beta_M$ of $R^+$. Using the Poincar\'{e}--Birkhof--Witt theorem it is clear that $W(\lambda)$ is spanned by elements of the form $X_1X_2\cdots X_M\bu(\lie h[t]^\tau)w_\lambda$ where each $X_p$ is either a constant or a monomial in the elements $\{(x^-_{\beta_p}\otimes t^s): s\in a_j\bz_+-\boa_j(\beta_p)\}$.  The length  of each $X_r$ is  bounded by  Lemma \ref{wfinite} and   equation  \eqref{gar2} proves that for any $\gamma\in R^+$ and $r\in\bz_+$, the element
$(x^-_\gamma\otimes t^{ra_j-\boa_j(\gamma)})\bu(\lie h[t]^\tau)w_\lambda$ is in the span of elements $\{(x_\gamma^-\otimes t^{sa_j-\boa_j(\gamma)})\bu(\lie h[t]^\tau)w_\lambda: 0\le s\le N\}$ for some $N$ sufficiently large.  An obvious induction on the length  of the product of monomials  shows that the values of $s$ are  bounded for each $\beta$ and the proof is complete.
\begin{rem}  Notice  that  the preceding argument  proves that the set $\wt W(\lambda)$ is finite. This is not obvious since $\wt W(\lambda)$ is not a subset of $\lambda-Q^+_0$.\end{rem}

\subsection{}\label{section3loc} Let  $\lambda\in P_0^+ $. Given any maximal ideal $\bi$ of $\ba_\lambda$ we define the local Weyl module, $$W(\lambda, \bi)= W(\lambda)\otimes_{\ba_\lambda} \ba_\lambda/\bi.$$  It follows from Proposition \ref{fingen} that $W(\lambda,\bi)$ is a finite--dimensional $\lie g[t]^\tau$--module in $\ti{\cal I}$ and $\dim W(\lambda,\bi)_\lambda=1$. A standard argument now proves that $W(\lambda,\bi)$ has a unique irreducible quotient which we denote as $V(\lambda,\bi)$. Moreover, $W(\lambda,\mathbf{I}_0)$ is a $\bz_+$--graded $\lie g[t]^\tau$--module and
 \begin{equation}\label{grirr} V(\lambda,\bi_0)\cong \ev_0^* V_{\lie g_0}(\lambda),\end{equation}  where   $\ev^*_0V$ is the representation of $\lie g[t]^\tau$ obtained by pulling back a  representation  $V$ of $\lie g_0$.

\subsection{}\label{canss} We now construct an explicit family of representations of $\lie g[t]^\tau$ which will be needed for our further study of $\ba_\lambda$.  
Given non--zero scalars $z_1,\dots, z_k$ such that $z_r^{a_j}\ne  z_s^{a_j}$ for all $1\le r\ne s\le k$  we have a canonical surjective morphism $$\lie g[t]^{\tau}\to \lie g_0\oplus \lie g^{\oplus k}\to 0,\ \ (x\otimes t^r)\to (\delta_{r,0}x, z_1^rx,\dots, z_k^rx).$$

 Given a representation $V$ of $\lie g$ and $z\ne 0$, we let $\ev_z^*V$ be the corresponding pull--back  representation of $\lie g[t]^\tau$; note that these representations are cyclic $\lie g[t]^\tau$--modules.
  Using the recursive formulae for $P_{\alpha,r}$ it is not hard to see that the following hold in the module $\ev^*_z V_{\lie g}(\lambda)$, $\lambda\in P^+$ and $\ev_0^*V_{\lie g_0}(\mu)$, $\mu\in P_0^+$:
\begin{gather*}\lie n^+[t]v_\lambda=0\ \  P_{i,r}v_\lambda= \binom{\lambda(h_{i })}{r}(-1)^rz^{a_jr}v_\lambda,\ \  \ i\in I,\ \ r\in\bn\\ \lie n^+[t]^\tau v_\mu=0,\
\ \ P_{i,r}v_\mu=0,\ \ i\in I,\ \ r\in\bn.
\end{gather*}

 The preceding discussion together with equation \eqref{grouplike} now proves the following result.
\begin{prop}\label{irr1}  Suppose  that $\lambda_1,\dots,\lambda_k\in P^+$ and $\mu\in P_0^+ $. Let  $z_1,\dots, z_k$ be non--zero complex numbers  such that $z_r^{a_j}\ne  z_s^{a_j}$ for all $1\le r\ne s\le k$. Then $$\ev^*_0 V_{\lie g_0}(\mu)\otimes\ev^*_{z_1}V_{\lie g}(\lambda_1)\otimes\cdots\otimes \ev^*_{z_k}V_{\lie g}(\lambda_k)$$ is an irreducible $\lie g[t]^\tau$--module. Moreover, $$\lie n^+[t]^\tau(v_\mu\otimes v_{\lambda_1}\otimes\cdots\otimes v_{\lambda_k})=0,\ \ \left(P_{i,r}-\pi_{i,r}\right)(v_\mu\otimes v_{\lambda_1}\otimes\cdots\otimes v_{\lambda_k})=0,\ \ i\in I,\ \  r\in\bz_+,$$ where $$\sum_{r\in\bz_+}\pi_{i,r}u^r=\prod_{s=1}^k(1-z_s^{a_j}u)^{\lambda_s(h_i)},\ \ i\in I.$$\hfill\qedsymbol
\end{prop}

\begin{rem} In particular, the modules constructed in the preceding proposition are modules of the form $V(\lambda, \bi)$ where $\lambda=\mu+\lambda_1+\cdots+\lambda_k$.  
The converse statement is also true; this follows from the work of \cite{NSS}. An independent proof can be deduced once we complete our study of $\ba_{\lambda}$.
\end{rem}
%%%%%%%%%%%%%%%%%%%%%%%%%%%%%%%%%%%%%%
%%%%%%%%%%%%%%%%%%%%%%%%%%%%%%%%%%%%%%

\section{The algebra \texorpdfstring{$\ba_\lambda$}{A} as a Stanley--Reisner ring}\label{section4}
 For the rest of this section we denote by $\text{Jac}({\ba_{\lambda}})$ the Jacobson radical of $\ba_{\lambda}$, and use freely the fact that the Jacobson radical of a finitely--generated commutative algebra coincides with its nilradical.
\subsection{} The main result of this section is the following.
\begin{thm} \label{alambda}
The algebra   $\ba_\lambda/{\rm{Jac}}(\ba_\lambda)$ is isomorphic to the algebra $\tilde\ba_\lambda$ which is the quotient  of $\bu(\lie h[t]^{\tau})$ by the ideal generated by the elements \begin{equation}\label{firstset}P_{i,s},\ \ i\in I(j),\ \  s\ge \lambda(h_i)+1,\end{equation} and \begin{equation}\label{secondset} P_{1,r_1}\cdots P_{n,r_n},\ \  \sum_{i=1}^n\boa_i^\vee (\alpha_0)r_i>\lambda(h_0).\end{equation} Moreover, ${\rm{Jac}}(\ba_\lambda)$ is generated by the images of the elements in  \eqref{secondset} and ${\rm{Jac}}(\ba_\lambda) = 0 $  if $\boa_j^\vee (\alpha_0) = 1$.
\end{thm}

\begin{ex} We discuss the statement of the theorem in the case of $(B_n, D_n)$. Recall that in this case, $h_0 = h_{n-1} + h_n$ and so $\boa_j^{\vee}(\alpha_0) = 1$. Thus, ${\rm{Jac}}(\ba_\lambda) = 0 $ and \eqref{secondset} becomes \[ P_{n-1, r_{n-1}}P_{n,r_n} , \ r_{n-1} + r_n > \lambda(h_0).\]
\end{ex}

 Before proving \thmref{alambda} we note several interesting consequences.

\subsection{} \label{stanleyreisner}

We recall the definition of a Stanley--Reisner ring, and the correspondence between Stanley--Reisner rings and abstract simplicial complexes  (for more details, see \cite{FMS2014}). 

Given a monomial $m = x_{i_1} \cdots x_{i_\ell}$ we say that $m$ is squarefree if $ i_1 < \cdots < i_\ell.$  We say an ideal of $\mathbb{C}[x_1, \dots, x_n]$ is a squarefree monomial ideal if it is generated by squarefree monomials.  A quotient of a polynomial ring by a squarefree monomial ideal is called a Stanley--Reisner ring.

We now prove the following consequence of \thmref{alambda}.

\begin{prop} \label{alambdasr} The algebra $\ba_\lambda/{\rm{Jac}}(\ba_\lambda)$ is a Stanley--Reisner ring with Hilbert series  $$\mathbb{H}(\ba_\lambda/{\rm{Jac}}(\ba_\lambda))=\sum_{\sigma\in \Sigma_{\lambda}}\prod_{P_{i,r}\in \sigma}\frac{t^{a_jr}}{1-t^{a_jr}},$$
where $\Sigma_{\lambda}$ denotes the abstract simplicial complex corresponding to $\ba_\lambda/{\rm{Jac}}(\ba_\lambda)$. Moreover, if $\boa_j^{\vee}(\alpha_0) = 1, $  the Krull dimension of $\ba_{\lambda}$  is given by $$d_\lambda=\lambda(h_0) + \sum_{i: \boa_i(\alpha_0)=0}\lambda(h_i). $$ If in addition we have $|\{i : \boa_i (\alpha_0) >0\}|=2,$ then the algebra $\ba_{\lambda}$ is Koszul and Cohen-Macaulay.

\end{prop}

\begin{ex} In the case  $(B_n, D_n)$, we have since  $\alpha_0 = \alpha_{n-1} + 2 \alpha_n$ and  $h_0=h_{n-1}+h_ n$  that  $\ba_\lambda$ is Koszul and Cohen--Macauley.  
%and 
%$$H(\Delta_{\lambda},t)=\frac{\min\{\lambda(h_0),\lambda(h_{n-1})\}t+1}{(1-t)^{\sum_{i=0}^{n-2}\lambda(h_i)+\lambda(h_0)}}.$$
\end{ex}

To see how \propref{alambdasr} follows from Theorem \ref{alambda}, we need to understand  the Stanley--Reisner rings in terms of abstract simplicial complexes.  

\subsection{}

Let $k \in \mathbb{N}$ and let $X= \{ x_1, \dots , x_k \}$. An abstract simplicial complex $\Sigma$ on the set $X$ is a collection of subsets of $X$ such that if $A \in \Sigma$ and if $B \subset A$, then $B \in \Sigma.$ There is a well known correspondence between abstract simplicial complexes, and ideals in $\mathbb{C}[X] = \mathbb{C}[x_1, \dots , x_k]$ generated by squarefree monomials which is given as follows: if  $\Sigma$ is an abstract simplicial complex, let $\bj_{\Sigma} \subset \mathbb{C}[X]$ be the ideal generated by the elements of the set $$\{ x_{i_1} \cdots x_{i_r} \ | \ 1 \leq r \leq k, \{x_{i_1} , \dots, x_{i_r} \} \notin \Sigma  \}.$$ The following proposition can be found in \cite{FMS2014}.

\begin{prop} \label{srcorr}  Given any abstract simplicial complex, $\Sigma$ the ring  $ \mathbb{C}[X] / \bj_{\Sigma}$ is a Stanley--Reisner ring.  Conversely, any Stanley--Reisner ring is isomorphic to  $ \mathbb{C}[X] / \bj_{\Sigma}$ for some $X = \{x_1, \dots, x_k \}$ and some abstract simplicial complex $\Sigma$ on $X$.
\hfill\qed
\end{prop}

\subsection{} If $A \in \Sigma$, we call $A$ a simplex, and a simplex of $\Sigma$ not properly contained in another simplex of $\Sigma$ is called a facet. Let $\mathcal{F}(\Sigma)$ denote the set of facets of $\Sigma.$ For sets $B\subset A$, we have the Boolean interval $[B,A]=\{C: B\subset C \subset A\}$ and let $\bar{A}=[\emptyset,A]$. The dimension of $\Sigma$ is the largest of the dimension of its simplexes, i.e.
$$\text{dim} \Sigma=\max\{|A|: A\in \Sigma\}-1.$$
The simplicial complex $\Sigma$ is said to be pure if all elements of $\mathcal{F}(\Sigma)$ have the same cardinality. An enumeration $F_0,F_1,\dots, F_p$ of $\mathcal{F}(\Sigma)$ is called a shelling if for all $1\leq r\leq p$ the subcomplex
$$\left(\bigcup_{i=0}^{r-1}\bar{F_i}\right)\cap \bar{F_r} $$
is a pure abstract simplicial complex and $(\text{dim} F_r-1)$--dimensional.

The following can be found in \cite{FMS2014}.

\begin{prop}\label{prop54} If $\Sigma$ is pure and shellable, then the Stanley--Reisner ring of $\Sigma$ is Cohen--Macaulay.
\hfill\qed
\end{prop}

\subsection{} \label{stanleyreisnerlemma} It is immediate from \eqref{firstset} and \eqref{secondset} that $\ba_\lambda/{\rm{Jac}}(\ba_\lambda)$ is a Stanley--Reisner ring.  Let $\Sigma_{\lambda}$ denote the corresponding abstract simplicial complex. We have the following lemma.

\begin{lem}\label{shell}
Assume that $\boa_j^{\vee}(\alpha_0) = 1$ and $\{i : \boa_i (\alpha_0) >0\}=\{s,j\}$.
Then the simplicial complex $\Sigma_\lambda$ is pure and $\{F_0,\dots,F_{\min\{\lambda(h_0),\lambda(h_{s})\}}\}$ defines a shelling, where
$$F_r=\left(\prod_{\substack{i : \boa_i(\alpha_0)=0\\ 1\leq r_i \leq \lambda(h_i)}} P_{i,r_i}\right) P_{j,1}\cdots P_{j,\lambda(h_0)-r}P_{s,1}\cdots P_{s,r},\ \ 0\leq r\leq \min\{\lambda(h_0),\lambda(h_{s})\}.$$
\proof
Let $F$ a facet of $\Sigma_{\lambda}$, i.e., $F$ is not contained properly in another simplex of $\Sigma_{\lambda}$. It is clear that the cardinality of $F$ is less or equal to $d_{\lambda}$. If it is strictly less, then $P_{i,r}F$ is a face of $\Sigma_{\lambda}$ for some $i$ and $r$, which is a contradiction. Hence all facets have the same cardinality. The shelling property is straightforward to check.
\endproof
\end{lem}

\subsection{Proof of Proposition~\propref{alambdasr}} 
The statement of the Hilbert series and Krull dimension are immediate consequences of the correspondense between $\ba_\lambda/{\rm{Jac}}(\ba_\lambda)$ and $\Sigma_\lambda$, and can be found in \cite{FMS2014}. If $\boa_j^{\vee}(\alpha_0) = 1$ and $|\{i : \boa_i (\alpha_0) >0\}|=2,$  then it follows from Theorem \ref{alambda} that $\ba_\lambda$ is a   quotient of a polynomial algebra by a quadratic monomial ideal, and hence  Koszul (see \cite{Fr75}). Proposition~\ref{prop54} and Lemma~\ref{shell} together show that $ \ba_{\lambda} $ is Cohen--Macaulay.

\subsection{} In this subsection, we note another  interesting consequence of \thmref{alambda}.
\begin{prop}\label{inforone} Let $\lambda\in P_0^+$. Then  $\ba_\lambda/\rm{Jac}(\ba_\lambda)$ is either infinite--dimensional or isomorphic to $\bc$. Moreover, the latter is true iff the following two conditions hold:
\begin{enumerit}
\item[(i)] for $i\in I(j)$,   we have   $\lambda(h_i)> 0$ only if $\boa_i^\vee(\alpha_0)>0$,
\item[(ii)]$\lambda(h_0) < \boa_i^\vee(\alpha_0)$ if $i=j$ or if   $i\in I(j)$ and $\lambda(h_i)>0$.
%\min\{\boa_j^{\vee}(\alpha_0),\ \boa_i^{\vee}(\alpha_0): i\in I(j),\ \boa_i(\alpha_0)\neq 0,\ \lambda(h_i)\neq 0\}$,
\end{enumerit}

\end{prop}

\begin{pf}  Suppose that $\lambda$ satisfies the conditions in (i) and (ii). To prove that $\dim \ba_\lambda/\rm{Jac}(\ba_\lambda)=1$   it suffices to prove that the elements $P_{i,s}\in\rm{Jac}(\ba_\lambda)$ for all $i\in I$ and $s\ge 1$.   Assume first that $i\ne j$. If $ \lambda(h_i)=0$ then  equation \eqref{pgamma} gives $P_{i,s}w_\lambda =0$ for all $s\ge 1$. If $\lambda(h_i)>0$ then the conditions imply that  $\lambda(h_0)<\boa_i^\vee(\alpha_0)$ and hence equation \eqref{secondset} shows that $P_{i,s}\in \rm{Jac}(\ba_\lambda)$ for all $s\ge 1$. If $i=j$ then again the result follows from \eqref{secondset} and condition (ii).

We now prove the  converse direction.  Suppose that (i) does not hold. Then, there exists $i\ne j$ with $\boa_i(\alpha_0)=0$ and $\lambda(h_i)>0$.  Equation \eqref{secondset}   implies that  the preimage of $\rm{Jac}(\ba_\lambda)$  is contained in the  ideal of $\bu(\lie h[t]^\tau)$   generated by the elements $\{P_{i,s}: i\in I, \boa^\vee _i(\alpha_0)>0\}$. 
Hence, using  Lemma \ref{poly}  we see that the  image of the elements $\{P_{i,1}^r:r\in\bn\}$ in $\ba_\lambda/\rm{Jac}(\ba_\lambda)$ must remain linearly independent showing  that the algebra is infinite--dimensional.

Suppose that (ii) does not hold.  Then either $\lambda(h_0)\geq \boa_j^\vee(\alpha_0)$ or $\lambda(h_0)\ge \boa_i^\vee(\alpha_0)$ for some $i\in I(j)$ with $\lambda(h_i)> 0$. In either case   \eqref{secondset} and Lemma \ref{poly} show that the  image of the elements $\{P_{i,1}^r:r\in\bn\}$ in $\ba_\lambda/\rm{Jac}(\ba_\lambda)$ must remain linearly independent showing  that the algebra is infinite--dimensional.\end{pf}
\begin{cor} The algebra  $\ba_\lambda$ is  finite--dimensional  iff it is a local ring. It follows that $W(\lambda)$ is finite--dimensional iff $\ba_\lambda$ is a local ring.
\end{cor}
\begin{pf}
If  $\ba_\lambda$ is finite--dimensional then so is $\ba_\lambda/\rm{Jac}(\ba_\lambda)$ and the corollary is immediate from the proposition. Conversely suppose that  $\ba_\lambda$ is a local ring. By the proposition and equation \eqref{pgamma}, we have  $$P_{i,s}w_\lambda=0,\ \ {\rm{if}}\ \ \boa_i^\vee(\alpha_0)=0, \ \ s\in\bn.$$  If $\boa_i^\vee(\alpha_0)\ne 0$ we still have from \eqref{pgamma} that $P_{i,s}w_\lambda=0$ if $s$ is sufficiently large. Otherwise,  equation \eqref{secondset} shows that   there exists $N\in\bz_+$ such that $$P_{i,s}^N w_\lambda =0,\ \  {\rm{for\ all}}\ \ \ i\in I,  \  \ s\in\bn.$$  This proves that $\ba_\lambda$  is   generated by finitely many nilpotent elements and since it is a commutative algebra it is finite--dimensional.  The second statement of the corollary is now immediate from Proposition \ref{fingen}.
\end{pf}

\subsection{}  We turn to the proof of Theorem \ref{alambda}. It follows from equation \eqref{pgamma} that the elements in \eqref{firstset} map to zero in $\ba_\lambda$. Until further notice, we shall prove results which are needed to show that  the elements in \eqref{secondset} are in $\rm{Jac}(\ba_\lambda)$.

 Given $\alpha,\beta\in R$,  with $\ell\alpha+\beta\in R$, let  $c(\ell,\alpha,\beta) \in \mathbb{Z}\backslash\{0\}$ be such that $$ \text{ad}_{x_{\alpha}}^\ell(x_{\beta})  = c(\ell, \alpha, \beta) x_{\ell \alpha + \beta}.$$ The following is trivially checked by induction.

\begin{lem}\label{relat1} Let $\gamma\in\Delta$ and $\beta\in R^+ \setminus \Delta$ be such that $\beta+\gamma\notin R$  and  $(\beta,\gamma)>0.$  Given $m,n,s,p,q \in\bz_+$ we have 
$$(x^+_{\gamma}\otimes t^{p})^{(s+d_\gamma q)}(x^+_{\beta-\gamma}\otimes t^{m})^{(s)} (x^-_\beta\otimes t^{n})^{(q+s)}=C (x^-_{s_\gamma(\beta)}\otimes t^{n+d_\gamma p})^{(q)}(x^+_{\gamma}\otimes t^{p})^{(s)}  (x^-_{\gamma}\otimes t^{m+n})^{(s)} +X$$ where $X\in \bu(\lie g[t]^{\tau}) \lie n^+[t]^{\tau}$  and $ {(d_{\gamma}!)^q}C =c(d_\gamma,\gamma,-\beta)^q c(1,\beta-\gamma,-\beta)^s .$
\hfill\qedsymbol \end{lem}

It is immediate that under the hypothesis of the Lemma we have for all $P\in\bu(\lie h[t]^\tau)$ that  \begin{gather} \label{step}(x^+_{\gamma}\otimes t^{p})^{(s+ d_\gamma q)}(x^+_{\beta-\gamma}\otimes t^{m})^{(s)} (x^-_\beta\otimes t^{n})^{(q+s)} P w_\lambda\\ \notag=   C(x^-_{s_\gamma(\beta)}\otimes t^{n+ d_\gamma p})^{(q)}(x^+_{\gamma}\otimes t^{p})^{(s)}(x^-_{\gamma}\otimes t^{m+n})^{(s)}P w_\lambda, \end{gather} for some $C\ne 0$.

\subsection{}\label{betas}
Recall  that given any root $\beta\in R^+$ we can choose $\alpha\in\Delta$ with $(\beta,\alpha)>0$. Moreover if $\beta\notin\Delta$ and $\beta$ is long then $\beta+\alpha\notin R$. Setting  $\alpha_{i_0}=\alpha_j,\ \ \beta_0=\alpha_0,$  we set $\beta_1=s_{i_0}\beta_0$ and note that $\beta_1\in R^+$. If $\beta_1\notin\Delta$ then we choose $\alpha_{i_1}\in\Delta$ with $(\beta_1,\alpha_{i_1})>0$ and set $\beta_2=s_{i_1}\beta_1$. Repeating this if neccessary we reach a stage when $k\ge 1$ and $\beta_k\in\Delta$. In this case we set $\alpha_{i_k}=\beta_k$.
 We claim that  
\begin{equation}\label{anzahl}|\{0\leq r \leq k : i_r=i\}|=\boa_i^{\vee}(\alpha_0),\ \ 1\leq i\leq n.\end{equation}
 To see this,  notice that since the $\beta_p$ are long roots, we have $h_{\beta_p}=h_{\beta_{p-1}}-h_{i_{p-1}}$. Hence, $$h_0=\sum_{s=0}^k h_{i_s}=\sum_{i=1}^n \boa^\vee_i(\alpha_0)h_i.$$ Equating coefficients gives \eqref{anzahl}.

\subsection{} Retain the notation of Section \ref{betas}. We now prove that 
\begin{equation}\label{first1}P_{i_{k},s_{k}}\cdots P_{i_{0},s_0}w_\lambda=0,\ \ \text{if } \ (s_0+\cdots +s_k)\ge\lambda(h_0)+1. \end{equation}
We begin with the equality $$ w= (x^-_{0}\otimes 1)^{(s_0+\cdots+s_k)}w_\lambda=0,\ \ \    (s_0+\cdots +s_k)\ge \lambda(h_0) + 1,$$ which is  a defining relation for $W(\lambda)$.  
Recalling that $j=i_0$ and setting  $$X_1=(x_j^+\otimes t)^{(s_0 +d_{\alpha_j}(s_1+\cdots +s_k))}(x^+_{\alpha_0-\alpha_j}\otimes t^{a_j-1})^{(s_0)}$$ we get by applying \eqref{step} 
$$0=X_1w= (x^-_{\beta_1}\otimes t^{d_{\alpha_j}}  )^{(s_1+\cdots +s_k)}P_{i_0,s_0}w_{\lambda}.$$ More generally, if we set $$X_{r+1}=(x^+_{\alpha_{i_{r}}}\otimes t^{\delta_{i_r, j}})^{(s_{r}+d_{\alpha_{i_r}}(s_{r+1}+\cdots+r_{k}))}(x^+_{\beta_{r}-\alpha_{i_{r}}}\otimes t^{m_{r}})^{(s_{r})},$$
where $m_r = a_j-\delta_{i_r,j}- d_{\alpha_j}|\{0 \leq q <r \ | \ i_q=j\}|$ we find after repeatedly applying \eqref{step} that
$$0=(x^+_{\beta_{k}}\otimes t^{\delta_{i_k,j}})^{(s_k)}X_k\cdots X_1w=P_{i_{k},s_{k}}\cdots P_{i_{0},s_0}w_\lambda=0.$$ This  proves the assertion.

\subsection{} We can now prove that  $$P_{1, r_1}\cdots P_{n,r_n}\in {\rm Jac}(\ba_\lambda)\ \ {\rm{if}}\ \  \sum_{i=1}^n\boa_i^\vee(\alpha_0) r_i>\lambda(h_0).$$  Taking $s_p=r_m$ whenever $i_p=m$ in \eqref{first1} and using \eqref{anzahl} we see that
\begin{equation}\label{first2}P_{1, r_1}^{\boa_1^\vee(\alpha_0)}\cdots P_{n,r_n}^{\boa_n^\vee(\alpha_0)}w_\lambda=0\ \    {\rm{if}}\ \  \sum_{i=1}^n\boa_i^\vee(\alpha_0) r_i>\lambda(h_0).\end{equation} Multiplying through by appropriate powers of $P_{i,r_i}$, $1\le i\le n$ we get that for some $s\ge 0$ we have  $$P_{1,r_1}^s\cdots P_{n,r_n}^sw_\lambda =0,\ \    {\rm{if}}\ \  \sum_{i=1}^n\boa_i^\vee(\alpha_0) r_i>\lambda(h_0).$$ Hence  $P_{1,r_1}^s\cdots P_{n,r_n}^s= 0$ in $\ba_\lambda$ proving that $P_{1, r_1}\cdots P_{n,r_n}\in {\rm Jac}(\ba_\lambda)$. This argument proves that there exists a well--defined morphism of algebras
\begin{equation}\label{surjectivemapa}\varphi:  \tilde\ba_\lambda \twoheadrightarrow \ba_\lambda/\text{Jac}(\ba_\lambda). \end{equation} 

We now prove,
\begin{lem}\label{factordi}
If  $\boa_j^{\vee}(\alpha_0)=1$ the map $\varphi$ factors through $\ba_\lambda$, i.e., we have a commutative diagram
$$
\begin{tikzpicture}[column sep=small]
\node (A) at (-1,0) {$\tilde\ba_\lambda$};
\node (B) at (1.7,0) {$\ba_\lambda/{\rm{Jac}}(\ba_\lambda)$};
\node (C) at (1.7,-1.7) {$\ba_{\lambda}$};
\draw
(A) edge[->>,>=angle 90] node[left] {} (B)
(A) edge[->>,>=angle 90] node[left] {} (C)
(C) edge[->>,>=angle 90] node[left] {} (B);
\end{tikzpicture}
$$
\end{lem}
\begin{proof} Using \eqref{first2} it suffices to prove that if 
 $\boa_j^{\vee}(\alpha_0)=1$ then
\begin{equation*}\boa_i^{\vee} (\alpha_0) \leq 1 \ \ \forall i \in I.\end{equation*}
Since $\boa_j(\alpha_0)=a_j\geq 2>\boa_j^{\vee}(\alpha_0)=1$ we see that $\lie g$ cannot be of simply laced type and hence $\alpha_j$ is short. It follows that $s_{\alpha_0}\alpha_j=\alpha_j-\alpha_0$ is also short and so $h_{\alpha_0 - \alpha_j} =d_jh_0- h_j$. If $\boa_i^{\vee}(\alpha_0) > 1$ for some $i \not= j,$  then we would have 
\begin{equation*} \boa_{i}^{\vee}(\alpha_0 - \alpha_j) = d_j \boa_i^{\vee}(\alpha_0) \geq 2d_j.\end{equation*} Since $\alpha_j$ is short this is impossible unless $\lie g$ is of type $F_4$ and $j=4$. This case can be handled by an inspection.

\end{proof}

\subsection{} Using Lemma \ref{factordi} and \eqref{surjectivemapa} we see that the proof of Theorem \ref{alambda} is complete if we show that the map \eqref{surjectivemapa} is injective. Since $\tilde\ba_\lambda$
is a quotient of $\bu(\lie h[t]^{\tau})$ by a square--free ideal, it has no
nilpotent elements and thus $\text{Jac}(\tilde \ba_\lambda) = 0.$ So if $f$ is a nonzero element in $\tilde\ba_\lambda$, there exists a maximal ideal $\tilde{\bi}_{f}$ of $\tilde\ba_\lambda$ so that $f \notin  \tilde{\bi}_{f}.$ Therefore, by Lemma \ref{poly} we can choose a tuple $(\pi_{i,r})$, $ i\in I$, $r\in \bn$ satisfying the relations \eqref{firstset} and \eqref{secondset} such that under the evaluation map sending $P_{i,r}$ to $\pi_{i,r}$ the element $f$ is mapped to a non--zero scalar. Define $z_1,\dots,z_k$ and $\lambda_1,\dots,\lambda_k\in P^+$ by 
$$\pi_i(u)=1+\sum^{}_{r\in\bn}\pi_{i,r}u^r=\prod_{s=1}^k(1-z_s^{a_j}u)^{\lambda_s(h_i)},\ \ i\in I$$
and set $\mu=\lambda-(\lambda_1+\dots+\lambda_k)\in P_0$. In what follows we show that $\mu\in P_0^+$. Since $(\pi_{i,r})$ satisfies the relations in \eqref{firstset} we have that $\mu(h_i)\in \bz_+$ for $i\in I(j)$. Moreover,  since $(\pi_{i,r})$ satisfies  \eqref{secondset} we get $\mu(h_0)\in \bz_+.$
To see this, note that the coefficient of $u^r$ in $\prod_{i\in I} \pi_i(u)^{\boa_i^{\vee}(\alpha_0)}$ is given by
\begin{equation}\label{cofur}\sum_{(r_{i_k})}\ \prod_{i\in I} \prod_{k=1}^{\boa^{\vee}_i(\alpha_0)} \pi_{i,r_{i_k}},\end{equation}
where the sum runs over all tuples $(r_{i_k})$ such that 
$\sum_{i\in I} \sum_{k=1}^{\boa^{\vee}_i(\alpha_0)}r_{i_k}=r.$ Set $r_i=\text{max}\{r_{i_k},\ 1\leq k\leq \boa^{\vee}_i(\alpha_0)\}$, $ i\in I$ and observe that if $r>\lambda(h_{0})$, then
$$\sum_{i\in I}\boa_i^{\vee}(\alpha_0)r_i\geq r>\lambda(h_{0})$$
and hence \eqref{cofur} vanishes. It follows that
$$\mu(h_0)=\lambda(h_0)-\text{deg}(\prod_{i\in I} \pi_i(u)^{\boa_i^{\vee}(\alpha_0)})\in \bz_+.$$

Now using Proposition \ref{irr1} we have a quotient of $W(\lambda)$ where $f$ acts by a non--zero scalar on the highest weight vector. Hence $f^N\notin \text{Ann}_{\lie h[t]^{\tau}}(w_{\lambda})$ for all $N\geq 1$, i.e.  the image of $f$ under the map \eqref{surjectivemapa} is non--zero. This proves the map
\eqref{surjectivemapa} is injective, and so Theorem \ref{alambda} is established.

%%%%%%%%%%%%%%%%%%%%%%%%%%%%%%%%%%%%%%%%%%%%%%%%%%%%%%%%%%%%%%%%%%%%%%%%%%%%%%%%%
%%%%%%%%%%%%%%%%%%%%%%%%%%%%%%%%%%%%%%%%%%%%%%%%%%%%%%%%%%%%%%%%%%%%%%%%%%%%%%%%%%%%%%%%%

\section{Finite--dimensional global Weyl modules}\label{section5}
In this section we  give necessary and sufficient conditions for the global Weyl module to be finite--dimensional.

\subsection{}  Recall the elements $\theta_k\in R^+$, $0\le k< a_j$  defined in Section \ref{delta0}.
\begin{thm}\label{criterion} Given  $\lambda\in P_0^+$,
 the module $W(\lambda)$  is an irreducible $\lie g[t]^\tau$--module and hence isomorphic to $\ev^*_0 V_{\lie g_0}(\lambda)$  iff  the following hold:\begin{equation}\label{condir1}\lambda(h_0)=0\  {\rm{and}}\ \ \lambda(h_i)>0\ {\rm{only\  if}}\ \ \boa_i(\theta_{a_j-1})=\boa_i(\alpha_0).\end{equation}
 \end{thm}
The proof of the theorem can be found in the rest of the section.

\begin{ex}
We discuss the finite dimensional irreducible global Weyl modules for the example  $(B_n, D_n)$. In this case recall $\theta_1 = \alpha_1 + \dots + \alpha_n$  and $\alpha_0 = \alpha_{n-1} + 2 \alpha_n.$ Thus, $W(\lambda)$ is an irreducible $\lie g[t]^{\tau}$--module and hence isomorphic to $\text{ev}_0^*V_{\lie g_0}(\lambda)$ if and only if $\lambda = r \lambda_{n-1}$ for $ r \in \mathbb{Z}_+.$

\end{ex}

\subsection{} Suppose that $\lambda$ satisfies the conditions of the theorem. Notice that $$W(\lambda)\cong\ev_0^* V_{\lie g_0}(\lambda)\iff \lie g[t]^\tau[s]w_\lambda=0,\ \ s\in\bn.$$ Recall from Section \ref{proptheta} that $\theta_{a_j-1}+\alpha_j-\alpha_0=\sum_{i\in I(j)\cup\{0\}}p_i\alpha_i\in R_0^+\cup\{0\}$. If $\theta_{a_j-1}=\alpha_0-\alpha_j$, we have
\begin{align*}0=(x^{+}_{\alpha_j}\otimes t^{ra_j+1})(x^{-}_{\alpha_0}\otimes 1)w_{\lambda}&=(x^{-}_{\theta_{a_j-1}}\otimes t^{ra_j+1})w_{\lambda}.\end{align*}
Otherwise, $\theta_{a_j-1}+\alpha_j-\alpha_0\in R_0^+$ and it follows that if  $p_i>0$ then $\boa_i(\theta_{a_j-1}-\alpha_0)\not=0$ and hence by our assumptions on $\lambda$ we have $\lambda(h_i)=0$, i.e. 
 $(\lambda, \theta_{a_j-1}+\alpha_j-\alpha_0)=0$. Using the defining relations of $W(\lambda)$ we get $$(x^{-}_{\theta_{a_j-1}-\alpha_0+\alpha_j}\otimes 1)w_{\lambda}=0.$$ Since $\lambda(h_0)=0$ we now have 
\begin{align*}0=(x^{-}_{\theta_{a_j-1}-\alpha_0+\alpha_j}\otimes 1)(x^{+}_{\alpha_j}\otimes t^{ra_j+1})(x^{-}_{\alpha_0}\otimes 1)w_{\lambda}&=(x^{-}_{\theta_{a_j-1}-\alpha_0+\alpha_j}\otimes 1)(x^{-}_{\alpha_0-\alpha_j}\otimes t^{ra_j+1})w_{\lambda}&\\&=(x^{-}_{\theta_{a_j-1}}\otimes t^{ra_j+1})w_{\lambda}.\end{align*}
So in either case we found that $(x^{-}_{\theta_{a_j-1}}\otimes t^{ra_j+1})w_{\lambda}=0$.
By Propostion \ref{facts} and the discussion in Section \ref{proptheta} we know that 
 $\lie g_1$ is an irreducible $\lie g_0$--module generated by $x_{\theta_{a_j}-1}^-$ by applying elements $x_i^+$, $i\in I(j)\cup\{0\}$ and so
$$(\lie g_1\otimes t\mathbb{C}[t^{a_j}])w_{\lambda}=0.$$ 
Assume that $(\lie g_m\otimes t^m\bc[t^{a_j}])w_{\lambda}=0$ for all $m$  with $1\leq m <k \leq a_j$. Since $1\leq k-m<k$, we also have $(\lie g_{k-m}\otimes t^{k-m}\bc[t^{a_j}])w_{\lambda}=0$ by our induction hypothesis. Now by Proposition~\ref{facts} we have $\lie g_k=[\lie g_{k-m},\lie g_m]$ if $k< a_j$ and $\lie g_k=[\lie g_{1},\lie g_{a_j-1}]$ if $k=a_j$ and hence using the induction hypothesis we obtain 
$$(\lie g_k\otimes t^k\bc[t^{a_j}])w_{\lambda}=0,$$  proving that $W(\lambda)$ is irreducible.

\subsection{} We prove the forward direction of Theorem \ref{criterion} in the rest of the section for which we need some additional results.
\begin{lem}\label{munuirred} Let $\mu,\nu\in P_0^+$  and assume that $W(\nu)$ is reducible. Then $W(\mu+\nu)$ is also reducible.
\end{lem}
\begin{pf}   It is easily seen using the defining relations of $W(\mu+\nu)$ that we have  a map of $\lie g[t]^\tau$--modules  $W(\mu+\nu)\rightarrow \text{ev}_0^{*} V_{\lie g_0}(\mu)\otimes W(\nu)$  extending the assignment $w_{\mu+\nu}\to v_\mu\otimes w_\nu$.  Since $W(\nu)$ is reducible there 
exists $x\in \lie g[t]^\tau[s]$, $s\geq 1$ with  $xw_{\nu}\neq 0$. Since $xv_\mu=0$, we now  get $x(v_{\mu}\otimes w_{\nu})=v_{\mu}\otimes xw_{\nu}\neq 0$. Hence $0\ne xw_{\mu+\nu}\in W(\mu+\nu)[s]$ and the result follows.
\end{pf}
\subsection{}\label{findimimp} 
 Suppose that $\lambda,\mu\in P_0^+$ are such that  there exists $0\ne \Phi\in
 \text{Hom}_{\lie g_0}(\lie g_1\otimes V_{\lie g_0}(\lambda),V_{\lie g_0}(\mu))$. Then $V:=V_{\lie g_0}(\lambda)\oplus V_{\lie g_0}(\mu)$ admits a $\lie g[t]^\tau$--structure extending the canonical $\lie g_0$--structure as follows:
$$(x\otimes 1)(v_1,v_2)=(xv_1,xv_2),\ \  (y\otimes t)(v_1,v_2)=(0, \Phi(y\otimes v_1)),\ \ \lie g[t]^\tau[s](v_1,v_2)=0,\ \  s\geq 2,$$
where $(v_1,v_2)\in V$,\ $x\in \lie g_0$ and $y\in\lie g_1$. It is easily checked that if  $\lambda-\mu\in Q^+$  then  $V$ is a quotient of the global Weyl module $W(\lambda)$. 

\begin{prop}\label{funirr1} The global Weyl module $W(\lambda_i)$ is not irreducible if $i=0$ or $i\in I(j)$ with $\boa_i(\theta_{a_j-1})\neq \boa_i(\alpha_0)$.
\end{prop}
\begin{proof}  Recall that  $w_{\circ}$ is the longest element in the Weyl group defined by the simple roots $\{\alpha_i: i\in I(j)\}$ and note that $w_\circ \theta_{a_j-1}\in R^+_{a_j-1}$. It follows that $\alpha_0+w_\circ\theta_{a_j-1}\notin R $ and hence $0\le w_\circ\theta_{a_j-1}(h_0)\le 1$. Setting  $\mu_0=\lambda_0-w_{\circ}\theta_{a_j-1}$ we see that 
$\mu_0(h_i)=-w_{\circ}(\theta_{a_j-1})(h_i)\geq 0$ for $i\in I(j)$ and $ \mu_0(h_0)=1-w_{\circ}\theta_{a_j-1}(h_0)\geq 0$, i.e. $\mu_0\in P_0^+$.  Since $\lie g_{1}$ is an irreducible $\lie g_0$--module of lowest weight $-\theta_{a_j-1}$, the PRV theorem \cite{Mat89}, \cite{Ku88}  implies that $V_{\lie g_0}(\mu_0)$ is a direct summand of $\lie g_{1}\otimes V_{\lie g_0}(\lambda_0)$. It follows  from the discussion preceding the proposition that $W(\lambda_0)$ is not irreducible.

 It remains to consider a node $i\in I(j)$ with $\boa_i(\theta_{a_j-1})\neq \boa_i(\alpha_0)$. By way of contradiction suppose that  $W(\lambda_i)$ is irreducible. Using Proposition\ref{inforone}(ii)  we can assume that  $\boa_i(\alpha_0)>0$. Moreover,  by Section \ref{delta0} we know that $\theta_{a_j-1}-\alpha_0+\alpha_j\in R^+_0$ and hence  $\boa_i(\theta_{a_j-1})>\boa_i(\alpha_0)> 0$. If $\lie g$ is of classcial type, this is only possible if $\boa_i(\theta_{a_j-1})=2,\ \boa_i(\alpha_0)=1$ and hence we have a pair of roots of the form
$$\alpha_0=\cdots+\alpha_i+\cdots+2\alpha_{j}+\cdots,\ \ \theta_{a_j-1}=\cdots+2\alpha_i+\cdots+\alpha_{j}+\cdots$$
which is a contradiction. In other words, there is no such node if $\lie g$ is of classical type. If $\lie g$ is of exceptional type, a case by case analysis shows that $W(\lambda_i)$ is not irreducible using the discussion preceding the proposition.
\end{proof}
\subsection{}
\iffalse  Suppose that   $W(\lambda)$ is irreducible. The isomorphism in  \eqref{grirr} gives  $$W(\lambda)\cong\ev^* V_{\lie g_0}(\lambda)$$ and hence  $W(\lambda)$ is
finite--dimensional.  Since $W(\lambda)_\lambda\cong\ba_\lambda$ the algebra  $\ba_\lambda$ is  a finite--dimensional algebra. Proposition \ref{inforone} implies that ${\rm{Jac}}(\ba_\lambda)$ is an ideal of codimension one and hence we have $$\lambda(h_0)<\boa_j^\vee(\alpha_0)$$ and  for $i\ne j$, 
\begin{equation}\label{irrung1}\lambda(h_i)>0\implies \boa_i^\vee(\alpha_0)>0 \ {\rm{and}}\ \  \lambda(h_0)<\boa_i^\vee(\alpha_0).\end{equation}\fi
 Assume that $\lambda$ violates one of the conditions in \eqref{condir1}, i.e. $\lambda(h_i)>0$ where $i=0$ or $i\in I(j)$ and $\boa_i(\theta_{a_j-1})\neq \boa_i(\alpha_0)$. Now setting $\mu=\lambda-\lambda_i$ and $\nu=\lambda_i$ in Lemma \ref{munuirred}  and using Proposition \ref{funirr1} we see that $W(\lambda)$ is not irreducible which completes the proof of Theorem \ref{criterion}.

%%%%%%%%%%%%%%%%%%%%%%%%%%%%%%%%%%%%%%%%%%%%%%%%%%%%%%%%%%%%%%%%%%%%%%%%%%%%%%%%%
%%%%%%%%%%%%%%%%%%%%%%%%%%%%%%%%%%%%%%%%%%%%%%%%%%%%%%%%%%%%%%%%%%%%%%%%%%%%%%%%%%%%%%%%%

%%%%%%%%%%%%%%%%%%%%%%%%%%%%%%%%%%%%%%%%%%%%%%%%%%%%%%%%%%%%%%%%%%%%%%%%%%%%%%%%%
%%%%%%%%%%%%%%%%%%%%%%%%%%%%%%%%%%%%%%%%%%%%%%%%%%%%%%%%%%%%%%%%%%%%%%%%%%%%%%%%%%%%%%%%%
\section{Structure of local Weyl modules}

Recall from \secref{section2} that the equivariant map algebra $\lie g[t]^\tau$ is not isomorphic to an equivariant map algebra where the group $\Gamma$ acts freely on $A$.  When $\Gamma$ acts freely on $A$, the finite dimensional representation theory of the equivariant map algebra is closely related to that of the map algebra $\lie g \otimes A$ (see for instance \cite{FKKS11}). We have already seen a major difference between the finite dimensional representation theory of $\lie g [t]^{\tau}$ and that of $\lie g [t]$. Specifically, in \secref{section5}  we showed that unlike in the case of the current algebra, the  global Weyl module for $\lie g [t]^\tau$ can be finite--dimensional and irreducible for nontrivial dominant integral weights. 

In this section we discuss the structure of  local Weyl modules for the case of $(B_n, D_n)$ where $\lambda$ is a multiple of a fundamental weight, in which case $\ba_{\lambda}$ is a polynomial algebra.  We finish the section by discussing the complications in determining the structure of local Weyl modules for an abritrary weight $\lambda \in P_0^+$ . The simplest example is the case of $\omega_{n-1}= \lambda_0+\lambda_{n-1}$. Note that in this case $\ba_{\omega_{n-1}}$ is not a polynomial algebra.

 \subsection{} Recall that we have a well established theory of local Weyl modules for the current algebra $\lie g[t]$.   Given $\lambda\in P^+$ we denote by  $W_{\loc}^{\lie g}(\lambda)$, $\lambda\in P^+$ the $\lie g[t]$--module generated by an element $w_\lambda$ and defining relations $$\lie n^+[t] w_\lambda=0, \ \ (h\otimes t^r)w_\lambda=\delta_{r,0}\lambda(h)w_\lambda=0,\ \ (x_i^-\otimes 1)^{\lambda(h_i)+1}w_\lambda=0.$$ We remind the reader that $\{\omega_i: 1\le i\le n\}$  is a set of fundamental weights for $\lie g$ with respect to $\Delta$. The following was proved in \cite{FoL07} and \cite{Na11}.
\begin{equation}\label{untw}\dim W_{\loc}^{\lie g}(\lambda)=\prod_{i=1}^n\dim \left(W_{\loc}^{\lie g}(\omega_i)\right)^{m_i},\ \ \lambda=\sum_{i=1}^nm_i\omega_i\in P^+.
 \end{equation}
We can clearly regard $W_{\loc}^{\lie g}(\lambda)$, $\lambda\in P^+$ as a graded $\lie g[t]^\tau$ module by restriction, however it is not the case that this restriction gives a local Weyl module for $\lie g [t]^\tau$. The relationship between local Weyl modules for $\lie g [t]^\tau$ and the restriction of local Weyl modules for $\lie g [t]$ is more complicated, as we now explain.

 \subsection{} Given $z\in\bc^\times$ we have a isomorphism of Lie algebras $\eta_z: \lie g[t]\to \lie g[t]$ given by $(x\otimes t^r)\to (x\otimes (t+z)^r)$ and let $\eta_z^*V$ be the  pull--back through this homomorphism of a representation $V$ of $\lie g[t]$.  Suppose that $V$ is such that there exists $N\in\bz_+$ with $(\lie g\otimes t^m)V=0$ for all $m\geq N$. Then $(\lie g\otimes (t-z)^m)\eta_z^*V=0$ for all $m\geq N$.
 In particular we can regard the module $\eta_z^*V$ as a module for the finite--dimensional Lie algebra $\lie g\otimes \bc[t]/(t-z)^N$. Following \cite{FKKS11}, since $z\in\bc^\times$ we have  $$\lie g[t]/\lie g\otimes (t-z)^N\bc[t]\cong \lie g[t]^\tau/(\lie g\otimes (t-z)^N\bc[t])^\tau,$$ so if $V$ is a cyclic module for $\lie g[t]$ then $\eta_z^*V$ is a cyclic module for $\lie g[t]^\tau$.  

We now need a general construction. Given any finite--dimensional cyclic $\lie g[t]^\tau$--module $V$ with cyclic vector $v$  define an increasing filtration of $\lie g_0$--modules $$0\subset V_0=\bu(\lie g[t]^\tau)[0]v\subset\cdots\subset V_r=\sum_{s=0}^r\bu(\lie g[t])^\tau[s]v\subset\cdots\subset  V.$$ The associated graded space $\gr V$  is naturally a graded module for $\lie g[t]^\tau$ via the action $$(x \otimes t^s)\overline{w} = \overline{(x \otimes t^s)w}, \ \overline{w} \in V_r/V_{r-1} .$$   Suppose that $v$ satisfies the relations $$\lie n^+[t]^\tau v=0,\ \  (h\otimes t^{2k})v=d_k(h) v,\ \ \ \  d_k(h)\in\bc, \ \ k\in\bz_+, \ \ h\in\lie h.$$ Then since $\dim V<\infty$ it follows that $d_0(h)\in \bz_+$; in particular there exists $\lambda\in P_0^+$ such that $d_0(h)=\lambda(h)$ and a simple checking shows that $\gr V$ is a quotient of $W_{\loc}(\lambda):=W(\lambda,\mathbf{I}_0)$.

The following is now immediate.
 \begin{lem} \label{gradedlowerbound} Let $\lambda\in P^+$ and $z\in\bc^\times$.  The $\lie g[t]^\tau$--module $\gr\left(\eta^*_z W_{\loc}^{\lie g}(\lambda)\right)$ is a quotient of $W_{\loc}(\lambda)$ and hence $$\dim W_{\loc}(\lambda)\ge \dim W_{\loc}^{\lie g}(\lambda).$$\hfill\qedsymbol
 \end{lem}

\subsection{} For the rest of this section, we consider the case of $(B_n, D_n)$, and study local Weyl modules corresponding to weights $r\lambda_i \in P_0^+$, where $r \in \mathbb{Z}_+$, and $0 \leq i \leq n-2$ (the $i=n-1$ case is discussed in \secref{section5}, where these local Weyl modules are shown to be finite--dimensional and irreducible). We remind the reader that $\lambda_0=\omega_n$, $\lambda_i=\omega_i$, $1\leq i\leq n-2$ and $\lambda_{n-1}=\omega_{n-1}-\omega_n$.  In particular, we show the reverse of the  inequality in \lemref{gradedlowerbound}, which proves the following proposition.  

\begin{prop}\label{fundmult} Assume that $(\lie g,\lie g_0)$ if of type  $(B_n,D_n)$.   For $0\le i\le n-2$  and $r\in\bz_+$  we have an isomorphism of $\lie g[t]^\tau$--modules $$W_{\loc}(r\lambda_i)\cong \gr\left(\eta^*_z W_{\loc}^{\lie g}(r\lambda_i)\right).$$
\end{prop}

 \subsection{} We recall standard results for local Weyl modules for the current algebra $\lie g[t]$.
 \begin{prop} \label{locweylcurrent}\begin{enumerit} 
 \item[(i)] Let $x,y,h$ be the standard basis for $\lie{sl}_2$ and set $y\otimes t^r=y_r$,  For $\lambda\in P^+$ the local Weyl module $W_{\loc}^{\lie{sl}_2}(\lambda)$ has basis $$ \{w_\lambda,\  \ y_{r_1}\cdots y_{r_k}w_\lambda: 1\le k\le \lambda(h),\ \   0\le r_1\le \cdots\le  r_k\le \lambda(h)-k\}.$$
 Moreover, $y_sw_{\lambda}=0$ for all $s\geq \lambda(h)$.
 \item[(ii)] Assume that $\lie g$ is of type $B_n$ (resp. $D_n$) and  assume that $i \ne n$ (resp. $i\ne n-1, n$). Then $$W_{\loc}^{\lie g}(\omega_i)\cong_{\lie g} V_{\lie g}(\omega_i)\oplus V_{\lie g}(\omega_{i-2})\oplus\cdots\oplus V_{\lie g}(\omega_{\bar i}), $$ where $$V_{\lie g}(\omega_{\bar i})  = V_{\lie g}(\omega_1),\ \ i\ {\rm{odd}},\ \ V_{\lie g}(\omega_{\bar i}) = \bc,\ \ i\ \ {\rm{even}}.$$
\item[(iii)] Assume that $\lie g$ is of type $B_n$ (resp. $D_n$), and let $i=n$ (resp. $i \in \{ n-1, n \}$).  Then $$W_{\loc}^{\lie g}(\omega_i)\cong_{\lie g} V_{\lie g}(\omega_i).$$
 \end{enumerit}\hfill\qedsymbol
 \end{prop}
 We remind the reader that $$\dim V_{\lie g}(\omega_i)=\begin{cases}\binom{2n+1}{i},\ \ \lie g= B_n, \ i \ne n\\ \\ \binom{2n}{i},\ \ \lie g=D_n, \  i \ne n-1, n. \end{cases}$$  Moreover, if  $\lie g$ is of type $B_n$, $$\dim V_{\lie g}(\omega_n)= 2^n,$$ and if $\lie g$ is of type $D_n$ and $ i \in \{n-1 , n \}$, then $$\dim V_{\lie g}(\omega_i)= 2^{n-1}.$$ 

\subsection{} 
Our goal is to prove that $$\dim W_{\loc}^{\lie g}(r\lambda_i)\ge \dim W_{\loc}(r\lambda_i),\ \ r\in\bn.$$ The proof needs several additional results, and we consider the cases $1 \leq i \leq n-2$ and $i=0$ separately. 

Recall that $\lie g_0 [t^2] \subset \lie g[t]^\tau,$ and so $W_{\loc}(r\lambda_i)$ can be regarded as a $\lie g_0[t^2]$--module by pulling back along the inclusion map  $\lie g_0 [t^2]  \hookrightarrow \lie g[t]^\tau.$ For ease of notation we  denote the element $w_{r\lambda_i}$ by $w_r$.  

\begin{lem}\label{teil1} \begin{enumerit}   \item[(i)] For $1 \leq i \leq n-2, $ $W_{\loc}(r\lambda_i)$ is generated as a  $\lie g_0[t^2]$--module by $w_r$ and $Yw_r$ where $Y$ is a monomial in the the elements $$ (x^-_{p,n}\otimes t^{2s+1})w_r,\ \ p\le i,\ \ 0\le s< r.$$

\item[(ii)] $W_{\loc}(r\lambda_0)$ is generated as  a $\lie g_0[t^2]$--module by $w_r$ and $Yw_r$ where $Y$ is a monomial in the the elements $$ (x^-_{p,n}\otimes t^{2s+1})w_r,\ \ p\le n,\ \ 0\le s< r.$$
\end{enumerit}

\begin{proof} First, for $1 \leq i \leq n-2$ the defining relation $x^-_0 w_r=0$ implies that $$(x^-_0\otimes t^{2s})w_r= ( x^-_{n-1}\otimes t^{2s})w_r= (x^-_n\otimes t^{2s+1})w_r=0,\ \ s\ge 0.$$ Since $x_p^-w_r=0$ if $p\ne i$ it follows that \begin{equation}\label{derfa}(x^-_{p,n}\otimes t^{2s+1})w_r=0,\ \ s\ge 0,\ \  p>i.\end{equation} Observe also that  $$(x^-_i)^{r+1}w_r=0\implies (x^-_i\otimes t^{2s})w_r=0,\ \ s\ge r,$$ and hence we also have that $$(x^-_{p,n}\otimes t^{2s+1})w_r=0,\ \ s\ge r,\ \ p\le i.$$
 A simple application of the PBW theorem now gives (i).  

For the case $i=0,$ we have $$ (x^-_{k,p} \otimes t^{2s})w_r = 0, \ \ 1 \leq k \leq p \leq n-1, \ s \geq 0.$$ The relation $(x^-_0)^{s+1}w_r = 0$ for $ s \geq r$  implies that $$ (x^-_0 \otimes t^{2s})w_r =0, \ \ s \geq r $$ and so $$ (x^-_n \otimes t^{2s+1})w_r = 0, \ \ s \geq r. $$ Hence $$ (x^-_{p,n} \otimes t^{2s+1})w_r=0, \ \ 1 \leq p \leq n, \ s \geq r$$
and (ii) is now clear.  
\end{proof}
\end{lem}

 \subsection{} We now prove,
 \begin{lem}\label{teil2} \begin{enumerit}
  \item[(i)] For $1 \leq i \leq n-2$, suppose that $Y=(x^-_{p_1,n}\otimes t^{2s_1+1})\cdots (x^-_{p_k, n}\otimes t^{2s_k+1})$ where $p_1\le \cdots\le p_k\le i$. Then $Yw_r$ is in the $\lie g_0[t^2]$--module generated by elements $Zw_r$ where $Z$  is a monomial  in the elements $(x^-_{i,n}\otimes t^{2s+1})$ with $s\in \bz_+$.
\item[(ii)] For $i=0,$ suppose that $Y=(x^-_{p_1,n}\otimes t^{2s_1+1})\cdots (x^-_{p_k, n}\otimes t^{2s_k+1})$ where $p_1\le \cdots\le p_k\le n.$ Then $Yw_r$ is in the $\lie g_0[t^2]$--module generated by elements $Zw_r$ where $Z$  is a monomial  in the elements $(x^-_{n}\otimes t^{2s+1})$ with $s\in \bz_+$.

\end{enumerit}

  \begin{pf} First, let $1 \leq i \leq n-2$. The proof proceeds by an induction on $k$. If $k=1$ and $p_1<i$ then  by setting $$ x^-_{p_1,n} = [ x^-_{p_1, i - 1} , x^-_{i,n} ] $$ we have $$x^-_{p_1, i-1}(x^-_{i,n}\otimes t^{2s_1+1})w_r=  (x^-_{p_1,n}\otimes t^{2s_1+1})w_r,$$  hence induction begins.  For the inductive step,  we  observe that $$(x^-_{p_1,n}\otimes t^{2s_1+1})\bu(\lie g_0[t^2])\subset \bu(\lie g_0[t^2])\oplus\sum_{m\ge 0}\sum_{p=1}^n\bu(\lie g_0[t^2])(x^{\pm}_{p,n}\otimes t^{2m+1}),$$  and hence it suffices to  prove that for all $1\le p\le n$ and $Z$ a monomial in   $(x^-_{i,n}\otimes t^{2s+1})$ we have that $(x^{\pm}_{p,n}\otimes t^{2m+1})Zw_r$ is in the $\lie g_0[t^2]$--submodule generated by elements $Z'w_r$ where $Z'$  is a monomial in   $(x^-_{i,n}\otimes t^{2s+1})$.  Denote this submodule by $M$. We  give the proof only for $(x^{-}_{p,n}\otimes t^{2m+1})Zw_r$, since the other case is proven similarly. If $p=i$, there is nothing to prove and if $p>i$ we get
$$(x^{-}_{p,n}\otimes t^{2m+1})Zw_r=X+Z(x^{-}_{p,n}\otimes t^{2m+1})w_r,$$
for some element $X\in M$. Since $(x^{-}_{p,n}\otimes t^{2m+1})w_r=0$ by \eqref{derfa}, we are done. If $p<i$, we 
consider $$(x^-_{p, i-1} \otimes t^{2m})(x^-_{i,n} \otimes t)^{\ell + 1}w_r = A(x^-_{p,n} \otimes t^{2m+1})(x^-_{i,n} \otimes t)^{\ell}w_r + B(x^-_{p , \bar i} \otimes t^{2m+2})(x^-_{i,n} \otimes t)^{\ell - 1}w_r,$$
for some non--zero constants $A$ and $B$. Since $$ (x^-_{p, i-1} \otimes t^{2m})(x^-_{i,n} \otimes t)^{\ell + 1}w_r \in M, $$ and $$ (x^-_{p , \bar i} \otimes t^{2m+2})(x^-_{i,n} \otimes t)^{\ell - 1}w_r \in M ,$$ we have, $$(x^-_{p,n} \otimes t^{2m+1})(x^-_{i,n} \otimes t)^{\ell}w_r \in M.$$ In order to show $$(x^-_{p, n} \otimes t^{2m+1})(x^-_{i,n} \otimes t^{2r_1+1}) \cdots (x^-_{i,n} \otimes t^{2r_{\ell} + 1})  w_r \in M$$  we let $h \in \lie h$ with $[h, x^-_{p,n}] =0$ and $[h, x^-_{i,n}] \ne 0.$ Then  $$ (h \otimes t^{2s})(x^-_{p, n} \otimes t^{2m+1})(x^-_{i,n} \otimes t) \cdots (x^-_{i,n} \otimes t)  w_r \in M$$ for all $s \geq 0.$ An induction on $|\{ 1 \leq s \leq \ell \ : \ r_s \ne 0  \}|$ finishes the proof for $1 \leq i \leq n-2$.  The $i=0$ case is identical.

\end{pf}
\end{lem}   
  
  \subsection{} Observe that the Lie subalgebra $\lie a[t^2]$ generated  by the elements $x^\pm_i\otimes t^{2s}$, $s\in\bz_+$ is isomorphic to the current algebra $\lie{sl}_2[t^2]$. Hence $\bu(\lie a[t^2])w_r\subset W_{\loc}(r
  \lambda_i)$ is a quotient of the local Weyl module for $\lie a[t^2]$ with highest weight $r$ and we can use the results of Proposition \ref{locweylcurrent}(i).
  We now prove, 
  \begin{lem}\begin{enumerit}

\item[(i)] For $1 \leq i \leq n-2,$ as a $\lie g_0[t^2]$--module $W_{\loc}(r\lambda_i)$ is spanned by $w_r$ and  elements $$Y(i, \bos )w_r:=(x^-_{i,n}\otimes t^{2s_1+1})\cdots (x^-_{i,n}\otimes t^{2s_k+1})w_r,\ \ k\ge 1, \ \ \bos\in\bz_+^k,\ \ 0\le s_1\le \cdots\le s_k\le r-k.$$
\item[(ii)] For $i =0,$ as a $\lie g_0[t^2]$--module $W_{\loc}(r\lambda_i)$ is spanned by $w_r$ and  elements $$Y(n, \bos )w_r:=(x^-_{n}\otimes t^{2s_1+1})\cdots (x^-_{n}\otimes t^{2s_k+1})w_r,\ \ k\ge 1, \ \ \bos\in\bz_+^k,\ \ 0\le s_1\le \cdots\le s_k\le r-k.$$ \end{enumerit} \end{lem}
  \begin{pf} First, we consider the case $1 \leq i \leq n-2$. By Lemma~\ref{teil1} and  Lemma~\ref{teil2} we can suppose that $Y$ is an arbitrary monomial in the elements $(x^-_{i,n}\otimes t^{2s+1})$, $s\in\bz_+$. We proceed by induction on the length $k$ of $Y$. If $k=1$, then we have $$(x^-_{i,n}\otimes t^{2s+1})w_r=(x^-_{i+1, n}\otimes t)(x^-_i\otimes t^{2s})w_r=0, \ \ s\ge r,$$  by Proposition \ref{locweylcurrent}(i). This shows that induction begins.  Suppose now that $k$ is arbitrary and $\bos\in\bz_+^k$. Then,  by induction on $k$ \begin{equation} \label{idenin}(x^-_{i+1,n}\otimes t)^{k}(x^-_i\otimes t^{2s_1})\cdots (x^-_i\otimes t^{2s_k}) = A (x^-_{i,n}\otimes t^{2s_1+1})\cdots (x^-_{i,n}\otimes t^{2s_k+1}) +X  + Z,  \end{equation} where $A$ is a non--zero complex number and $X\in\sum_{m<k}\sum_{\bop\in\bz_+^{m}}\bu(\lie g_0[t^2])Y(i,\bop)$,  and $Z \in \bu(\lie g[t]^{\tau}) Y(i+1,\bos')$ and so $Zw_r=0$ .

 To see \eqref{idenin} we proceed by induction on $k$.  For the base case, we have $$(x^-_{i+1, n} \otimes t)(x^-_{i} \otimes t^{2s_1}) = (x^-_{i,n} \otimes t^{2s_1+1}) + (x^-_i \otimes t^{2s_1})(x^-_{i+1,n} \otimes t), $$ so induction begins. For the inductive step, we have $$(x^-_{i+1, n} \otimes t)^k(x^-_i \otimes t^{2s_1}) \cdots (x^-_i \otimes t^{2s_k}) = (x^-_{i+1, n} \otimes t)^{k-1}(x^-_{i+1, n} \otimes t)(x^-_i \otimes t^{2s_1}) \cdots (x^-_i \otimes t^{2s_k}) $$ $$ = (x^-_{i+1, n} \otimes t)^{k-1}\sum_{m=1}^k (x^-_i \otimes t^{2s_1}) \cdots \widehat{(x^-_{i} \otimes t^{2s_m})} \cdots(x^-_i \otimes t^{2s_k})(x^-_{i,n} \otimes t^{2s_m+1}) . $$ Applying the inductive hypothesis finishes the proof of \eqref{idenin}.

 To finish the proof of the Lemma for $1 \leq i \leq n-2$, we use \eqref{idenin} to write $$(x^-_{i,n} \otimes t^{2s_1+1}) \cdots (x^-_{i,n} \otimes t^{2s_k+1})w_r = (x^-_{i+1, n} \otimes t)^k(x^-_i \otimes t^{2s_1}) \cdots (x^-_i \otimes t^{2s_k})w_r - Xw_r.$$ The inductive hypothesis applies to $Xw_r$. By Proposition \ref{locweylcurrent} we can write $$ (x^-_{i+1, n} \otimes t)^k(x^-_i \otimes t^{2s_1}) \cdots (x^-_i \otimes t^{2s_k})w_r $$  as a linear combination of elements where $s_p\le r-k$.  Applying \eqref{idenin} once again to each summand finishes the proof for $1 \leq i \leq n-2$. 

The case $i=0,$ is similar, using the identity $$(x^-_{n}\otimes t^{2s+1})w_r=(x^+_{n-1, n}\otimes t)(x^-_0\otimes t^{2s})w_r=0, \ \ s\ge r,$$ for the induction to begin, and 
  $$(x^+_{n-1,n}\otimes t)^{k}(x^-_0\otimes t^{2s_1})\cdots (x^-_0\otimes t^{2s_k})w_r = A (x^-_{n}\otimes t^{2s_1+1})\cdots (x^-_{n}\otimes t^{2s_k+1})w_r $$ for the inductive step.

\end{pf}

  \subsection{} We now prove \propref{fundmult}, first for $1 \leq i \leq n-2.$ Fix an ordering on the elements $Y(i,\bos)w_r$, $\bos\in\bz_+^k$ and $s_p\le r-k$ as follows: the first element is $w_r$ and an element $Y(i,\bos)$ precedes $Y(i,\bos')$ if $\bos\in\bz_+^k$ and $\bos'\in\bz_+^m$ if either $k<m$ or $k=m$ and $s_1+\cdots + s_k>s_1'+\cdots+s_k'$ and let $u_1,\dots, u_\ell$ be an ordered enumeration of this set.
  Denote by $U_p$ the $\lie g_0[t^2]$--submodule of $W_{\loc}(r\lambda_i)$ generated by the elements $u_m$, $m\le p$. It is straightforward to see that we have an increasing filtration of $\lie g_0[t^2]$--modules:$$0=U_0\subset U_1\subset\cdots\subset U_\ell= W_{\loc}(r\lambda_i).$$ Moreover $U_p/U_{p-1}$ is a quotient of the local Weyl module for $\lie g_0[t^2]$ with highest weight $(r-i_p)\omega_i+i_p\omega_{i-1}$ (we understand $\omega_0=0$), if $u_p= Y(i,\bos)$, $\bos\in \bz_+^{i_p}$. Using equation \eqref{untw} and Proposition \ref{locweylcurrent}(ii) we get
$$\dim U_p/U_{p-1}\le \left(\sum_{s=0}^i\binom{2n-1}{s}\right)^{r-i_p}\left(\sum_{s=0}^{i-1}\binom{2n-1}{s}\right)^{i_p}.$$  Summing we get
\begin{align*}\dim W_{\loc}(r\lambda_i)& \le \sum_{s=0}^r\binom{r}{s}\left(\sum_{s=0}^i\binom{2n-1}{s}\right)^{r-s}\left(\sum_{s=0}^{i-1}\binom{2n-1}{s}\right)^{s}&\\&
=\left(\binom{2n}{i}+\binom{2n}{i-1}+\cdots+\binom{2n}{1}\right)^{r}.\end{align*} 
For the $i=0$ case, $U_p/U_{p-1}$ is a submodule of the local Weyl module for $\lie g_0[t^2]$ with highest weight $(r-2i_p)\omega_n+i_p\omega_{n-1}=(r-i_p)\lambda_0+i_p\lambda_{n-1}$, if $u_p= Y(n,\bos)$, $\bos\in \bz_+^{i_p}$. Using equation \eqref{untw} and Proposition \ref{locweylcurrent}(iii) we get
  $$\dim U_p/U_{p-1}\le (2^{n-1})^{r-i_p}(2^{n-1})^{i_p}.$$  Summing we get
  $$\dim W_{\loc}(r\lambda_i)\le \sum_{s=0}^r\binom{r}{s}(2^{n-1})^{r-s}(2^{n-1})^{s}= (2^{n-1}+2^{n-1})^r = (2^n)^r .$$  Since we have already proved that the reverse equality holds the proof of \propref{fundmult} is complete

\subsection{Concluding Remarks}\label{notfreealambda}
We discuss briefly the structure of the local Weyl modules when $\lambda\in P_0^+$ is not a  multiple of a fundamental weight and  such that $\ba_\lambda$ is a proper quotient of a polynomial algebra.  The simplest example is the case of $(B_3, D_3)$ and $\lambda=\lambda_0+\lambda_2$, 
where we have  $$\ba_{\lambda} = \mathbb{C}[P_{2,1}, P_{3,1}]/(P_{2,1}P_{3,1}). $$ Given  $a \in \mathbb{C}^{\times}$ let $\bi_{(a,0)}$ denote the maximal ideal corresponding to $(P_{2,1} -a, P_{3,1})$ and for $b \in \mathbb{C}$ $\bi_{(0,b)}$ denote the maximal ideal corresponding to $(P_{2,1} , P_{3,1} - b)$. In the first case, the local Weyl module $W(\lambda, \bi_{(a,0)} )$ is a pullback of a local Weyl module for the current algebra $\lie g [t]$ and so $$ \text{dim }W(\lambda, \bi_{(a,0)}) = 22. $$ In the second case  the local Weyl module  $W(\lambda, \bi_{(0,b)} )$ is an extension of the  pullback of a local Weyl module for the current algebra by an irreducible $\lie g_0$--module, and it can be shown that $$ \text{dim }W(\lambda, \bi_{(0,b)}) = 32. $$ (see \cite{mattthesis} for details).  In particular the dimension of the local Weyl module depends on the choice of the ideal and hence the global Weyl module is not projective and hence not free as an $\ba_{\lambda}$--module. 
However, we observe the following: If we decompose the variety corresponding to $\ba_{\lambda}$ into irreducible components $X_1 \cup X_2$, where
$$X_1=\{(a,0) : a\in \bc\},\ \ X_2=\{(0,b) : b\in \bc\},$$
we see that the dimension of the local Weyl module is constant along $X_2$. So pulling back $W(\lambda)$ via the algebra map
$$\varphi:\ba_{\lambda}\rightarrow \ba_{\lambda},\ P_{2,1}\mapsto 0,\ P_{3,1}\mapsto P_{3,1}$$
we see that $\varphi^{*} W(\lambda)$ is a free $\mathbb{C}[P_{3,1}]$--module, where we view $\mathbb{C}[P_{3,1}]$ as the coordinate ring of $X_2$.
In general, preliminary calculations do show that in the case when $\ba_{\lambda}$ is a Stanley--Reisner ring there are only finitely many possible dimensions and that the dimension is constant along a suitable irreducible subvariety, i.e. the global Weyl module is free considered as a module for the coordinate ring $\mathcal{O}(X)$ of a suitable irreducible subvariety $X$.

%%%%%%%%%%%%%%%%%%%%%%%%%%%%%%%%%%%%%%%%%%%%%%%%%%%%%%%%%%%%%%%%%%%
\bibliographystyle{plain}
\bibliography{bibfile}

\end{document}